\documentclass[article]{amsart}

\usepackage{amsmath, amsthm, amsfonts,stmaryrd,amssymb}
\usepackage[top=1in, bottom=1.25in, left=1.25in, right=1.25in]{geometry}
\usepackage{mathrsfs}    
\usepackage{ dsfont }                  
\usepackage{mathtools}                   
\usepackage{booktabs}
\usepackage{tikz}
\usepackage{todonotes}
\usepackage[all]{xy}
\usepackage[utf8]{inputenc}
\usepackage{mathtools}                     
\usepackage{todonotes}
\usepackage{faktor}
\usepackage{ marvosym }
\usepackage{bbm}                         
\usepackage{framed}
\usepackage{graphicx}   
\usepackage{tikz}
\usetikzlibrary{positioning}
\usetikzlibrary{arrows}

\usepackage{xcolor}
\newif\ifcomments
\commentstrue
\newcommand\comment[1]{%
  \ifcomments
  #1
  \else
  \fi}
\definecolor{mycolor}{RGB}{20, 20, 122}

\usepackage{theoremref}
\numberwithin{equation}{section}
\theoremstyle{plain}
\newtheorem{remark}{Remark}
\newtheorem{theorem}{Theorem}
\newtheorem{lemma}[theorem]{Lemma}
\newtheorem{corollary}[theorem]{Corollary}
\newtheorem{proposition}[theorem]{Proposition}

\theoremstyle{definition}
\newtheorem{definition}{Definition}

\def\R{{\mathbb R}}

\let\on=\operatorname

\newcommand{\ud}{\,\mathrm{d}}

\newtheorem{rem}{Remark}

\begin{document}
\title[Second order models for optimal transport]{Second order models for optimal transport and cubic splines on the Wasserstein space}
\author{Jean-David Benamou}
\address{INRIA, Project team Mokaplan \\   Universit\'e Paris-Dauphine, PSL Research University, Ceremade  } 
\email{jean-david.benamou@inria.fr}
\author{Thomas O. Gallou\"et} 
\address{INRIA, Project team Mokaplan   \\  Universit\'e Paris-Dauphine, PSL Research University, Ceremade}
\email{thomas.gallouet@inria.fr}
\author{Fran\c{c}ois-Xavier Vialard}
\address{Universit\'e Paris-Dauphine, PSL Research University, Ceremade \\ INRIA, Project team Mokaplan}
\email{fxvialard@normalesup.org}

\maketitle

\begin{abstract}
On the space of probability densities, we extend the Wasserstein geodesics to the case of higher-order interpolation such as cubic spline interpolation. 
After presenting the natural extension of cubic splines to the Wasserstein space, we propose a simpler approach based on the relaxation of the variational problem on the path space. 
We explore two different numerical approaches, one based on 
 multi-marginal optimal transport and entropic regularization and the other based on semi-discrete optimal transport. 
\end{abstract}

\section{Introduction}

We propose a variational method to generalize cubic splines on the space of densities using multimarginal optimal transport. In short, the proposed method consists in minimizing, on the space of measures on the path space, under marginal constraints, the norm squared of the acceleration.  
\comment{ In this setting,  we show that  two numerical approaches, classical in optimal transportation can be applied. One is based on  entropic regularization  and the Sinkhorn Algorithm, the other relies on the 
Semi-Discrete formulation of Optimal Transportation and the computation of Laguerre cells, a classical problem in computationnal geometry.}  We showcase our methodology on 1D and 2D data. 
\par
In the past few years, higher-order interpolations methods have been investigated for applications in computer vision or medical imaging, for time-sequence interpolation or regression. The most usual setting is when data are modeled as shapes, which can be understood as objects embedded in the Euclidean space with no preferred parametrization: space of unparametrized curves or surfaces, or images are some of the most important examples. These examples are infinite dimensional but the finite dimensional case of a Riemannian manifold was interesting for camera motion interpolation as first introduced in \cite{Noakes1} and further developed in \cite{splinesCk,Crouch}. 
Motivated by different applications, the problem of interpolation between two shapes is usually treated via the use of a Riemannian metric on the space of shapes and computing a geodesic between the two shapes. From a mathematical point of view, shape spaces are often infinite dimensional and thus, non-trivial analytical questions arise such as existence of minimizing geodesics or global well-posedness of the initial value problem associated with geodesics. A finite dimensional approximation is still possible such as in \cite{TrVi2010}, in which spline interpolation is proposed for a diffeomorphic group action on a finite dimensional manifold. It has been  extended for invariant higher-order lagrangians in \cite{HOSplines1,HOSplines2} on a group, still finite dimensional. A numerical implementation of the variational and shooting splines has been developed in \cite{SinghVN15} with applications to medical imaging. The question of existence of an extremum is not addressed in these publications. An attempt is given in \cite{Vialard2016} where the exact relaxation of the problem is computed in the case of the group of diffeomorphisms of the unit interval. In a similar direction, in \cite{WirthSplines}, the authors discuss the convergence of the discretization of cubic splines in some particular infinite dimensional Riemannian context on the space of shapes.
\par
As a shape space, we are interested in this article in probability measures endowed with the Wasserstein metric. Since the Wasserstein metric shares some similarities with a Riemannian metric on this space of probability densities, it is natural to study further higher-order models in this context. Our motivation is to answer the following practical question of the extension of cubic splines to the Wasserstein space and their numerical computation. \par

We present in Section \ref{SecCubicSplines} the notion of cubic splines on a Riemannian manifold and detail its variational formulation in Hamiltonian coordinates. We then discuss independently in Section \ref{SecHomogeneousSpace} a geometric approach to the Wasserstein space that will be useful for the introduction of our proposed method detailed in Section \ref{modelspline}. Finally in Sections \ref{multimarg}  we present the numerical entropic relaxation method and an alternative numerical method based on semi-discrete optimal transport. The reader not interested in geometric interpretation can skip directly to Section \ref{modelspline}. \par 

To the best of our knowledge, this question has not been yet addressed in the literature on optimal transport until very recently in two independant and simultaneous  preprints~:  \cite{georgiouspline} and  \cite{ourpreprint} (this paper). Both work share the same  idea of relaxing the cubic spline formulation in the space of measure using multi-marginal optimal transport. Our paper however explores a larger hierarchy of models  and several numerical methods. 

\section{Cubic splines on Riemannian manifolds}\label{SecCubicSplines}
In this section, we present Riemannian cubics, which are the extension of variational splines to a Riemannian manifold $(M,g)$ where $g$ is the Riemannian metric.
Variational cubic splines on a Riemannian manifold are the minimizers of the acceleration; that is, denoting $\frac{D}{Dt}$ the covariant derivative, minimization on the set of curves $x: [0,T] \to M$ of the functional
\begin{equation}
\mathcal{E}(x) = \int_0^1 g(x)\left(\frac{D}{Dt} \dot{x},\frac{D}{Dt}\dot{x}\right) \ud t\,,
\end{equation}
subject to constraints on the path such as constraints on the tangent space, $(x(t_i),\dot{x}(t_i))$ are prescribed for a collection of times $t_i \in [0,1]$, or constraints on the positions such as $x(t_i) = x_i$. 
\par 
Under mild conditions on the constraints, if $M$ is complete, minimizers exist, for instance in the case of constraints on the tangent space mentioned above. A pathological case where minimizers might not exist is when the initial speed is not prescribed.  Consider for instance the two dimensional torus, where lines of irrational slopes are dense, it is possible to show that for any collection of points which do not lie on a line, the infimum of $\mathcal{E}$ is $0$ while it is never reached, see \cite{WirthSplines}. 
The Euler-Lagrange equation associated to the functional $\mathcal{E}$ is 
\begin{equation}\label{EqEulerLagrangeEquationSplines}
\frac{D^3}{Dt^3}\dot{x} - R\left(\dot{x},\frac{D}{Dt}\dot{x}\right)\dot{x} = 0\,,
\end{equation}
where $R$ is the curvature tensor of the Riemannian manifold $M$. Note that this equation is similar to a Jacobi field equation.
\par
We now formulate the variational problem in coordinates. In a coordinate chart around a point $x(t) \in M$, the geodesic equations are given by
\begin{equation} \label{RiemannGeodesics}
\frac{D}{Dt} \dot{x} = \ddot{x} +  \Gamma(x)(\dot{x},\dot{x})  = 0 \,,
\end{equation}
where $\Gamma$ is a short notation for the Christoffel symbols associated with the Levi-Civita connection. It is a second-order differential equation which is conveniently written as a first-order differential equation, via the Hamiltonian formulation.
Again in local coordinates on $T^*M$ the cotangent bundle of $M$, 
the geodesic equation can be written as
\begin{equation}\label{HamiltonianGeodesics}
\begin{cases}
\dot{p} + \partial_x H  = 0\\
\dot{x} - \partial_p H = 0\,,
\end{cases}
\end{equation}
where $H(x,p) = \frac 12 g(x)^{-1}(p,p)$. Note that, the ODE \eqref{RiemannGeodesics} can be obtained from the Hamiltonian system using $\dot{x} = g(x)^{-1}p$. 
From these two equivalent formulations \eqref{RiemannGeodesics} and \eqref{HamiltonianGeodesics}, it can be shown that $g^{-1}(x)(\dot{p} + \partial_x H) = \frac{D}{Dt} \dot{x} $.
Therefore, it proves that the variational spline problem can be rewritten in Hamiltonian coordinates as follows 
\begin{equation*}
\inf_{u} \int_0^1 g(x)^{-1}(a,a) \ud t \,,
\end{equation*}
under the constraint
\begin{equation*}
\begin{cases}
\dot{x} - g(x)^{-1}p = 0\\
\dot{p}  +\partial_x H(x,p) = a\,,
\end{cases}
\end{equation*}
with initial conditions $x(0) = x_0$ and $p(0) = p_0$.
It is natural to ask whether such variational problems carry over in infinite dimensional situations such as the Wasserstein space, which will be discussed in the rest of the paper.

\section{A formal application of spline interpolation to the Wasserstein space}\label{SecHomogeneousSpace}

It is well known that the Hamiltonian formulation of geodesics on the Wasserstein space, define over a riemannian manifold $M$,  are
\begin{equation} \label{Geodesics}
\begin{cases}
\dot{\rho} + \nabla \cdot (\rho \nabla \phi) = 0 \\
\dot{\phi} + \frac{1}{2}|\nabla \phi|^2  = 0\, ,
\end{cases}
\end{equation}
where $\rho: M \mapsto \R_{\geq 0}$ and $\phi: M \mapsto \R$  {\em implicitly   time dependant }   are respectively a probability density and a function. Note that these equations are valid when working with smooth densities. The Hamiltonian is the following,
\begin{equation} \label{Hamiltonian}
H(\rho,\phi) =  \frac{1}{2}  \int_M |\nabla \phi|^2 \rho \, d\mu_0 \,,
\end{equation}
where $\mu_0$ is a reference measure on $M$.

\begin{rem}
Taking the gradient of the equation governing $\phi$, and denoting $v=\nabla \phi$, we get Burger's equation:
\begin{equation} \label{Burger}
\dot{v} + (v,\nabla)v = 0 \,,
\end{equation}
where in coordinates, the operator $(v,\nabla)$ is defined as $(v,\nabla) w \doteq \sum_{i=1}^n v_i \nabla w_i$ where $v,w$ are vector fields and $n$ is the dimension of the $M$. In Lagrangian coordinates, this equation implies that
\begin{equation}
\ddot{\varphi} = 0 \,,
\end{equation}
where $\varphi (t) :$ $M\mapsto M$ is the Lagrangian flow associated with $v$ ($\dot{\varphi} = v \circ \varphi$), which is well-defined under sufficient regularity conditions.
\end{rem}

\begin{rem}
For the Wasserstein case, the operator is given by $g(\rho)^{-1}\phi = - \nabla \cdot [\rho \nabla \phi] $ so that the (formal) computation of the covariant derivative $\frac{D}{Dt} \dot{\rho}$ on the Wasserstein space is:
\begin{equation} \label{WassersteinConnection}
\frac{D}{Dt} \dot{\rho}= - \nabla \cdot [\rho \, (v + (v,\nabla)v)]\, ,
\end{equation}
where $v = \nabla \phi $ is the horizontal lift associated with $\dot{\rho}$, that is $\dot{\rho} + \nabla \cdot (\rho \nabla \phi) = 0 $. This result is proven rigorously in \cite{lott2008some}.
\end{rem}

From a control viewpoint, we aim at minimizing $\frac{1}{2}\int_0^1H(\rho,a) \, dt$ for the control system:
\begin{equation} \label{ControlledSystem}
\begin{cases}
\dot{\rho} + \nabla \cdot (\rho \nabla \phi) = 0 \\
\dot{\phi} + \frac{1}{2}|\nabla \phi|^2  = a\, ,\\
\end{cases}
\end{equation}
where $a$ is a time dependent function defined on $M$.
Alternatively, in terms of the variables $(\rho,\phi)$, this amounts to minimize
\begin{equation}\label{EqFirstProblem}
\int_0^1 \int_M | \nabla [\dot{\phi} + \frac{1}{2}|\nabla \phi|^2] |^2 \rho \, \ud \mu_0 \, \ud t \,,
\end{equation}
under the continuity equation constraint $\dot{\rho} + \nabla \cdot (\rho \nabla \phi) = 0 $.
It is a nonconvex optimization problem in the couple $(\rho,\phi)$. The key issue here is that the variational problem itself is a priori not well-posed since our formulation is valid in a smooth setting and to make it rigorous on the space of measures, the tight relaxation of this problem is needed. However, we do not address this issue in our work and in the next section we turn our attention to a simple relaxation of the problem which is probably not   tight.
\section{A hierarchy of relaxed models}\label{modelspline}

\subsection{Context} 

We recall the classical optimal transport setting. We have the following well known equivalence \cite{OttoPorousMedium,VilON}
\begin{equation}
\label{equalitystandart}
\begin{array}{ll} 

W_2^2\left(\rho_0,\rho_1\right) & \displaystyle  = \inf_{\varphi  }\int_0^1 \int_M | \dot{\varphi}|^2 \ud \mu_0 \ud t = 
\inf_{\rho ,v }\int_0^1 \int_M | v |^2 \ud \rho   \ud t \\[10pt]  & = \displaystyle  
\inf_{\rho }\int_0^1  \inf_{v } \int_M | v |^2 \ud  \rho   \ud t
=\inf_{\rho ,\nabla \phi  }\int_0^1   \int_M | {\nabla \phi  }|^2 \ud \rho   \ud t
\end{array} 
\end{equation}
Under  constraints that  $$  [ \varphi (t) ]_* \mu_0=\rho(t) \mbox{ for  }t=0,1$$ 
($   [ \varphi (t) ]_* \mu_0$  is the image measure of $\mu_0$~:  
$\int_M f(y)\ud   [ \varphi (t) ]_* \mu_0 (y)  = \int  f(T(x)) \ud \mu(x)$ for every measurable function $f: M \to \R$ ) \\
and  the continuity equation 
$$\dot{\rho  }+\nabla \cdot (\rho  v )=\dot{\rho} +\nabla \cdot (\rho  \nabla \phi )=0$$ 
with fixed initial and final conditions 
$$ \rho(0) = \rho_0  \mbox{ and }  \rho(1) = \rho_1. $$

Moreover, geodesics in the space of densities for the Wasserstein metric are given by \\
 $[ \varphi(t) ]_* \mu_0= \rho(t)$ and the associated displacement maps satisfy  $v \circ \varphi = \dot{\varphi }$. \\
 
 The last equality in \eqref{equalitystandart} exactly says that the  infimum  $\inf_{v(t)} \int_M | v(t) |^2 \ud  \rho(t)$ among all $v(t)$ satisfying the continuity equation at each time $t$ is achieved when $v(t)$ is a gradient.   This property is a consequence of a Riemannian submersion and $\nabla \phi$ is called the horizontal lift of $\dot{\rho}$. It is this last formulation that formally gives a Riemannian structure on the space of probability measures. See the remark \ref{gs} below for more details on the geometrical structure. \\
 
For higher-order variational problems, e.g. the minimization of the acceleration, the reduction in the last inequality does not holds true in general, even if the Riemannian submersion structure is present as shown in \cite{HOSplines2}.
 It means in the case of acceleration that, a priori, with the same constraint as for \eqref{equalitystandart}~: 
 \begin{equation}\label{equalitystandart2}
 \begin{array}{ll} 
\displaystyle \inf_{\varphi } \int_0^1 \int_M | \ddot{\varphi }|^2 \ud \mu_0 \ud t    & = 
\displaystyle \inf_{\rho ,v }\int_0^1 \int_M |\dot{v }+ (v ,\nabla)v |^2 \ud \rho    \ud t \\[10pt] &  \neq 
 \displaystyle \inf_{\rho ,\nabla \phi  }\int_0^1   \int_M | {\dot{\phi} +(\nabla \phi ,\nabla)\nabla \phi }|^2 \ud \rho  \ud t , 
 \end{array} 
\end{equation}
where we have used that $\ddot{\varphi }=\dot{v }\circ \varphi  + (v \circ \varphi ,\nabla)v \circ \varphi  $. \\

 
 \begin{remark}
 \label{gs}
 From a geometrical point of view, \eqref{equalitystandart} says the Wasserstein space can be seen, at least formally, as a homogeneous space as described in \cite[Appendix 5]{khesin2008geometry} and originally in \cite{OttoPorousMedium}.
  Consider the group of (smooth) diffeomorphisms of $M$ a closed manifold, $\on{Diff}(M)$, and the space of (smooth) probability densities $\on{Dens}(M)$.
   The space of densities is endowed with a $\on{Diff}(M)$ action defined by the pushforward, that is to a given $\varphi \in \on{Diff}(M)$ and $\rho \in \on{Dens}(M)$, the pushforward of $\rho$ by $\varphi$ is $\on{Jac}(\varphi^{-1})\rho \circ \varphi^{-1}$.
    By Moser's lemma, this action is transitive, thus making the space of densities as a homogeneous space.
   More importantly, there exists a compatible Riemannian structure between $\on{Diff}(M)$ and $\on{Dens}(M)$.
   Once having chosen a reference density $\mu_0$, the $L^2(M,\mu_0)$ metric on the diffeomorphism group descends to the Wasserstein $L^2$ metric on the space of densities, or in other words, the pushforward action $\varphi \mapsto \varphi_*\mu_0$ is a Riemannian submersion.
    An important property of Riemannian submersion is that geodesics on $\on{Dens}(M)$ are in correspondence with geodesics on the group, given by horizontal lift.
     This property is actually contained in Brenier's polar factorization theorem, which shows that the horizontal lift is the gradient of a convex function.
     
  \end{remark}   
  
  \subsection{The Monge formulation}
  
      In Section \ref{SecHomogeneousSpace} we used the formal Riemannian structure on the set of probability measure to define an intrinsic notion of splines,  \eqref{EqFirstProblem} is indeed  the RHS of inequality \eqref{equalitystandart2}. 
In this section we propose a simpler  alternative definition of Wasserstein splines based on the LHS of inequality \eqref{equalitystandart2}.


\begin{definition}[Monge formulation]\label{Mongeformulation}
Let $0 = t_0 < \ldots < t_n =1$, $n \geq 2$ and $\rho_1, \ldots, \rho_n $ be $n$ probability measures on $M$.
\par
Minimize, among time dependent maps $\varphi(t): M \mapsto M$,  
\begin{equation}\label{EqSecondProblem}
\int_0^1 \int_M | \ddot{\varphi} |^2 \ud \mu_0 \ud t \,,
\end{equation}
under the marginal constraints $\varphi(t_i)_* \mu_0 = \rho_i$.
This minimizing problem is denoted by $(MS)$.
\end{definition}
It is a Monge formulation of the variational problem, similar to standard optimal transport. 
On a Riemannian manifold $M$, the notation $\ddot{\varphi}$ stands for $\frac{D}{Dt}{\dot{\varphi}}$. 
By the change of variable with the map $\varphi$, the problem can be written in Eulerian coordinates, that is using the vector field associated with the Lagrangian map $\varphi$, $\partial_t \varphi = v \circ \varphi$, one aims at minimizing for $(\rho,u)$ 
\begin{equation}
\int_0^1 \int_M |u|^2 \rho \ud \mu_0 \ud t\,
\end{equation}
under the constraints
\begin{equation}
\begin{cases}
\dot{\rho} + \on{div} (\rho v) = 0 \\
\dot{v} + (v,\nabla)v = u\,,
\end{cases}
\end{equation}
with the marginals constraints $\rho(t_i) = \rho_i$.
\begin{remark}
Remark that formally when $v = \nabla \phi$, this new model reduces to the formulation \eqref{EqFirstProblem}. Therefore, it justifies the fact that Problem \eqref{EqSecondProblem} is a relaxation of \eqref{EqFirstProblem}. However, as already mentioned, this relaxation is probably not tight.
\par
Another formal geometric argument in the direction of proving that the two formulations are different is that the Wasserstein space has nonnegative curvature if the underlying space $M$ has nonnegative curvature, but the space of maps in the Euclidean space is flat. Therefore, the two Euler-Lagrange equations \eqref{EqEulerLagrangeEquationSplines} lead to a different evolution equations: for instance, if $M$ is the Euclidean space then the Euler-Lagrange equation for the second model is simply $\ddddot{\varphi} = 0$, which is a priori different from the splines Euler-Lagrange equation in the Wasserstein case.
\end{remark}
\subsection{The Kantorovich relaxation}
 Since, as is well-known in standard optimal transport, the Monge formulation is not well-posed for general given margins  $\rho_1, \ldots, \rho_n $, we propose  instead  to solve yet another relaxation of the problem on the space of curves which takes the form:

\begin{definition}[Kantorovich relaxation]\label{ThmKantorovich}
Let $0 = t_1 < \ldots < t_n =1$, $n \geq 3$ and $\rho_1, \ldots, \rho_n $ be $n$ probability measures on $M$.
\par
Minimize on the space of probability measures on the path space $H^2([0,1],M)$ denoted by $\mathcal{H}$ in short,
\begin{equation}\label{EqMultiMarginal}
\min_{\mu} \int_{\mathcal{H}} |\ddot{x}|^2 \ud \mu(x) \,,
\end{equation}
which is a linear functional  of  $ \ud \mu$.  The curves of densities is given by its marginals in time 
\begin{equation}\label{TimeMarginals}
t \mapsto  \rho(t) \mu_0 := [e_t]_*(\mu)\,,
\end{equation}
 $e_t$ is the evaluation function at time $t$~:
 if $\gamma \in H^2([0,1],M) \subset C^0([0,1],M)$ then $e_t(\gamma) = \gamma(t,.) \in M$.  \\
The notation $[e_t]_*\mu$ is the image measure by the map $e_t$ defined by duality~:\\ $\int_M f(y)\ud [e_t]_* \mu(y) = \int_{\mathcal{H}} f(e_t(x)) \ud \mu(x)$ for every measurable function $f: M \to \R$. 
Note that $x$ is a path on $[0,1]\times M$ while $y$ is a point on $M$.

With these notations, the marginal constraint at given time $t_i$ are 
\begin{equation}\label{EqMarginalConstraints}
[e_t]_*(\mu) = \rho_i \, \mu_0\,.
\end{equation}

\end{definition}
By standard arguments, the Kantorovich relaxation admits minimizers under general hypothesis on the manifold $M$, which we do not detail here. It is straightforward to check that existence of minimizers holds when $M = \R^d$.
\par
As expected, the Kantorovich formulation is the relaxation of the Monge formulation in Definition \ref{Mongeformulation}.
\begin{theorem}\label{ThmRelaxation}
Let $M = \R^d$, $0 = t_1 < \ldots < t_n =1$, $n \geq 3$ and $\rho_1, \ldots, \rho_n \in $ be $n$ probability measures on $\R^d$ with compact support 	and $\rho_1$ being atomless. Then, under the constraints \eqref{EqMarginalConstraints}, the infimums of the variational problem \eqref{EqSecondProblem} and \eqref{EqMultiMarginal} coincide, moreover, the infimum is attained for the latter.
\end{theorem}
\begin{proof}
See the proof of a more general result in Appendix \ref{Appendix}.
\end{proof}

First we remark that we can reformulate both the Monge and Kantorovich problems on the set of cubic splines. 
It is the purpose of the following lemmas and corollaries, whose proofs are straightforward. 
\begin{definition}[Cubic interpolant]
Let $(x_1,\ldots,x_n) \in \R^d$ be $n$ given points and $(t_1< \ldots<t_n)$ be $n$ timepoints. 
There exists a unique cubic spline minimizing the acceleration of the curve $x(t)$ such that $x(t_i) = x_i$. 
This unique curve is called cubic interpolant and is denoted by $c_{x_1,\ldots,x_n}$, depending implicitly on the timepoints.
\end{definition}

\begin{lemma}\label{ThmReduction}
When the supports of the measures $\rho_i$ are compact on $\R^d$, 
the support of every minimizing $\mu$ in Definition \ref{ThmKantorovich} is included in the set the cubic interpolants $c_{x_1,\ldots,x_n}$ for $(x_1,\ldots,x_n) \in \on{Supp}(\rho_1) \times \ldots \times \on{Supp}(\rho_n)$. 
\end{lemma}

\begin{proof}
The constraints are the marginal constraints $[e_{t_i}]_*(\mu) = \rho_i$ for $i \geq 3$ which implies that set of paths charged by an optimal measures satisfies $x(t_i) \in \on{Supp}(\rho_i)$. 
In particular, any path in this set can be replaced by its minimal spline energy, the cubic interpolant $c_{x_1,\ldots,x_n}$.
\end{proof}

\begin{corollary}
As a consequence, the set of paths charged by an optimal plan are uniformly $C^2$ and for every smooth function $\eta: \R^d \mapsto \R$ with compact support, the map $t \mapsto \langle \mu(t), \eta \rangle$ is $C^2$.
\end{corollary}

\begin{proof}
The set of cubic interpolants is compact since the map $(x_1,\ldots,x_n) \mapsto c_{x_1,\ldots,x_n}$ is continuous from $\R^{dn}$ to the space of $C^2$ fonctions (solution of an invertible linear system) and $\on{Supp}(\rho_i)$ are compact. Therefore, the set of maps are uniformly $C^1$. The last point follows directly.
\end{proof}

\begin{corollary}\label{ThmMultiMarginal}
The Kantorovich problem in Definition \ref{ThmKantorovich} on $\R^d$ reduces to a multimarginal optimal transport problem, as follows, let $c(x_1,\ldots,x_n)$ be the continuous cost of the cubic interpolant at times $t_1,\ldots,t_n$, the minimization of \eqref{EqMultiMarginal} reduces to the minimization of 
\begin{equation}\label{eqK}
\int_{M^n} c(x_1,\ldots,x_n) \ud \pi(x_1,\ldots,x_n) \phantom{1111111} \textrm{$(K)$}
\end{equation}
on the space of probability measures $\pi \in \mathcal{P}(M^n)$ and under the marginal constraints $(p_i)_*(\pi) = \rho_i$ where $p_i$ is the projection of the $i^{\text{th}}$ factor.
\end{corollary}
\begin{proof}
Direct consequence of Lemma \ref{ThmReduction}.
\end{proof}
Similarly 
\begin{corollary}\label{ThmMultiMarginalmonge}
The Monge problem in Definition \ref{Mongeformulation} on $\R^d$ reduces to a Monge multimarginal optimal transport problem, as follows, let $c(x_1,\ldots,x_n)$ be the continuous cost of the cubic interpolant at times $t_1,\ldots,t_n$, the minimization of \eqref{EqSecondProblem} reduces to the minimization of 
\begin{equation}\label{eqMmulti}
\int_{M} c\left(x, \varphi(t_1,x),\ldots,\varphi(t_n,x) \right) \ud \mu_0(x),
\end{equation}
on the space of path $\varphi \in C^2([0,1],M)$ (or even cubic splines) and under the marginal constraints $(\varphi(t_i))_*\mu_0 = \rho_i$.
 
 \end{corollary}

The dual formulation of the minimization problem $(K)$ is also well known \cite[Theorem 2.1]{PassKim2013}
\begin{definition}[Kantorovich dual problem $(KP)$]
Let  $ \mathcal{Q}=\left\{ \phi_i \in L^1(\rho_i \, \mu_0)\, , i=1..n \right\}$ be the space of integrable $n$-uplet. Maximize on $ \mathcal{Q}$ 
\begin{equation}\label{KP}
\sum_{i=1}^{n} \int_{M} \phi_i \rho_i \, \mu_0, \mbox{ under the constraint }\sum_{i=1}^{n}  \phi_i (x_i) \leq c(x_1,...,x_n).
\end{equation}
\end{definition}
And the following duality results holds true:
\begin{proposition}
There exists a $n$-uplet $(\phi_i)_{i=1..n}\in \mathcal{Q}$ optimal for $(KP)$. Moreover $(K)$=$(KP)$ and for any $\pi$ optimal in \eqref{eqK} there holds $\sum_{1}^{n} \phi_i (x_i) = c(x_1,...,x_n) $, $\pi$ almost everywhere.
\end{proposition}

A natural question is whether the solution of the Kantorovich problem $(K)$ is admissible in the Monge formulation $(MS)$ (Definition  \ref{Mongeformulation}). 
With the formulation reduced above the spline, given by \eqref{eqK} and \eqref{eqMmulti}, one can try to apply existing theory to answer to this question, see \cite{ PassKim2013, Passreview} and references therein for precise criterion.
However our cost does not satisfy any of those known criterion. 
In fact, we have the following result which proves that the relaxation to plans are necessary even in the context of Theorem \ref{ThmRelaxation}.
\begin{proposition}(Counter Example) \label{ThCounterExample}
Given the three-marginals problems of minimizing the acceleration, there exist data $(\rho_0,\rho_1,\rho_2)$ such that $\rho_0$ is atomless and such that the solution of $(K)$ is not a (measurable) Monge map.
\end{proposition}

\begin{proof}
Consider $\rho_0(x) = \mathbf{1}_{[-1,1]}$ and the Dirac masses $a = \delta_1$ and $b = \delta_{-1}$ and the maps $T_a,T_b$ that respectively pushforward $\rho_0$ onto $a$ and $b$. These maps are uniquely determined and affine. Consider now $\rho_2 = \frac 12 (T_a)_*\rho_0  + \frac 12 (T_b)_*\rho_0 = \frac a2 + \frac b2$. Then, introducing $(T^{1/2}) = \frac12 (\on{Id} + T)$, we consider 
$\rho_1 = \frac12 (T^{1/2}_a)_*\rho_0 + \frac12 (T^{1/2}_b)_*\rho_0$, note that it is equal to $\rho_0$ since the maps $T^{1/2}_{a,b}$ are affine. 
\par
By construction, the minimization of the acceleration for $(\rho_0,\rho_1,\rho_2)$ is null since it is a mixture of plans supported by straight lines. If there existed an optimal Monge solution it is necessarily supported by only one map denoted by $T$ and since the cost is null, the map at time $1/2$ is necessarily $T^{1/2}$ defined above. The preimage of $1$ (resp. $-1$) by $T$ is a measurable set $A$ (resp. $B$). Then, necessarily, $\rho_1 = (T^{1/2})_*\chi_A + (T^{1/2})_*\chi_B$, and in fact, $T_{|A} = T_a$ and  $T_{|B} = T_b$ (since the image of the map is known). Therefore, we have $\rho_1 = 2 \chi_A \circ (T_a^{1/2})^{-1} + 2 \chi_B \circ (T_b^{1/2})^{-1}$ which is not equal to the uniform Lebesgue measure on $[-1,1]$.
\end{proof}
\begin{figure}
\centering
\begin{tikzpicture}[>=triangle 45,font=\sffamily]
    \draw[->] (0,0) -- (5,0)node[anchor= north] {\text{time}};   
     \draw[->] (0,-3) -- (0,3)node[anchor=east] {\text{$\R$}};  
     \draw[thick,,line width=1mm,blue]  (-0.05,-2) -- (-0.05,2) ;
     \draw[thick,,line width=1mm,red]  (0.05,-2) -- (0.05,2) ;
     \draw[fill=red!40,opacity=0.5] (0,-2) --(0,2)--(4,2)--(0,-2);
     \draw[fill=blue!40,opacity=0.5] (0,-2) --(0,2)--(4,-2)--(0,-2);
     \node at (4,-2) {$\bullet$};
     \node at (4,-2.3) {$-1$};
      \node at (4,2) {$\bullet$};
     \node at (4,2.3) {$1$};
     \node at (4,0) {$\bullet$};
     \node at (4,-0.3) {$t = 2$};
     \node at (2,0) {$\bullet$};
     \node at (2,-0.3) {$t = 1$};
     \node at (0,0) {$\bullet$};
     \node at (-0.6,-0.3) {$t = 0$};
\end{tikzpicture}
\caption{The inital density at time $0$ is described with a mixture of two densities colored in red and blue which are evolving indepently along straight lines in time. The blue density is mapped onto $-1$ and the red density is mapped onto $1$. The acceleration cost is null and the proof of Proposition \ref{ThCounterExample} shows that it is not possible to reproduce the density at time $1/2$ by a map.}
\end{figure}
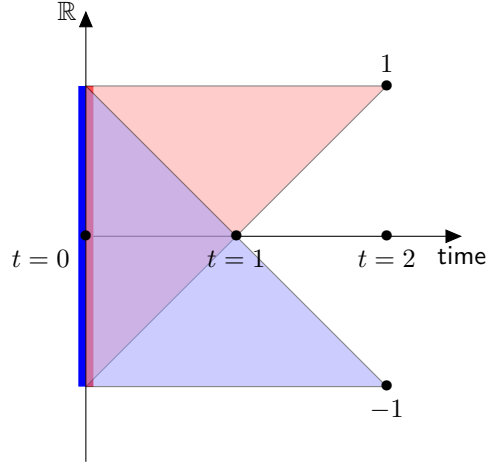

\begin{remark}
It is an open question to prove or disprove a similar result when the final density $\rho_2$ is atomless. 
The counterexample explained above strongly uses the fact that the final density is a sum of Dirac masses and it might not be robust when replacing the final density by a uniform density on a small interval.
\end{remark}

\subsection{The corresponding interpolation problem on the tangent space}\label{SecPhaseSpace}
The relaxed problem on the space of curves can be used to define variational interpolation problem on the phase space, or more precisely on the tangent space $TM$. Since the space $H^2([0,T],M)$ is contained in $C^1([0,T],M)$, one can formulate the optimal transport problem on phase space (identified with the tangent space) for the acceleration cost.
\begin{definition}[Optimal transport on phase space]\label{costphase}
Let $\bar{\rho_0},\bar{\rho_1} $ be two probability measures on $TM$.
Minimize on the space of probability measures on $\mathcal{H}$,
\begin{equation}\label{EqMultiMarginal2}
\min_{\mu} \int_{\mathcal{H}} |\ddot{x}|^2 \ud \mu(x) \,,
\end{equation}
which is a linear functional  of  $ \mu$ under the marginal constraints 
\begin{align}\label{EqMarginalConstraints2}
[j_0]_*(\mu) = \bar{\rho_0 }\,, \text{ and } [j_1]_*(\mu) =\bar{ \rho_1}\,,
\end{align}
where $j_t: H^2([0,T],M) \to TM$ is defined by $j_t(x) = (x(t),\dot{x}(t))$.
\end{definition}

\begin{proposition}[Optimal interpolation on phase space]\label{ThMongeInPhase}
The support of every optimal solution is contained in the set of cubic splines interpolating between $(x,v) \in \on{Supp}(\bar{\rho_0})$ and $(y,w) \in \on{Supp}(\bar{\rho_1})$. 
Moreover if $M = \R^d$ and if $ \bar{\rho_0}$ has density with respect to the Lebesgue measure, then the unique solution to Problem \eqref{EqMultiMarginal2} is characterized by a map $\varphi:$ $TM \mapsto TM$. 
\end{proposition}

Remark that the optimal solution in the last part of Proposition \ref{ThMongeInPhase} provides an interpolation on the phase space using $[j_t]_*(\mu)$.
\begin{proof}
The proof of the first part is similar to Lemma \ref{ThmReduction} and the second part follows by application of Brenier's theorem since the total cost of the cubic splines between $(x,v)$ and $(y,w)$ can be explicitly computed as
\begin{equation}\label{EqCostTM}
c_{ph}((x,v),(y,w)) = 12|x-y|^2 + 4(|v|^2 + |w|^2 + \langle v,w \rangle + 3 \langle v+w, x-y \rangle)\,
\end{equation}
and satisfies the twisted condition, so \cite[Theorem 10.28]{VilON} applies.
\end{proof}
Note that this problem is very different from using the Wasserstein distance on $\mathcal{P}(TM)$ where the tangent space $TM$ is endowed with the direct product metric. 
Indeed, the cost $c_{ph}$ does not vanish on the diagonal $(x,v) =(y,v)$ contrarily to the quadratic cost on $TM$.
\par
Interestingly, let us remark that the multimarginal problem can be recast as the minimization problem on $\Pi \in \mathcal{P}(\underbrace{TM \times \ldots \times TM}_{\text{n times}})$, denoting $\Pi_{t_i,t_{i+1}}$ the pushforward on $TM \times TM$ at times $(t_i,t_{i+1})$,
\begin{equation}
\min_{\pi} \sum_{i = 1}^{n-1} \langle \Pi_{t_i,t_{i+1}} ,c_{ph}((x_i,v_i),(x_{i+1},v_{i+1})) \rangle
\end{equation}
under the constraints that $[e_{t_i}]_*(\Pi_{i,i+1}) = \rho_i$. From the numerical point of view, this rewriting might be useful since the cost used on the multimarginal problem is now separable in time. This relaxation to the tangent space is used in the semidiscrete algorithm in Section \ref{SecImplementationSemiDiscrete}.
Obviously, up to the minimization on the variables $v_i$, we retrieve the minimization problem $(K)$ since one has a cost $c$ which is defined on $M^n$
\begin{equation}
\label{psc} 
c(x_0, \ldots, x_n) = \min_{v_0, \ldots ,v_n} \sum_{i=1}^{n-1} c_{ph}((x_i,v_i),(x_{i+1},v_{i+1}))
\end{equation}
where the index $i$ runs over the marginals. 

\section{Numerical Study }\label{multimarg}

\comment{We have discussed several variational relaxation of the  classical definition of splines, applied to the Wasserstein space of densities. 
 At least two different numerical techniques from Optimal Transportation  can be  used in this setting.   We apply the Entropic regularisation and Sinkhorn (briefly 
 recalled in appendix \ref{App2}  first to a simple Hermite interpolation problem (section \ref{Herm}) and then in section to the multimarginal problem (\ref{eqK}). 
 In section \ref{SD}, we use the semi-discrete Optimal Transportation approach in the spirit of  \cite{memi}  directly to problem (\ref{EqMultiMarginal}) 
 without the time discretisation in (\ref{eqK}).} 

\subsection{Hermite interpolation}
\label{Herm} 

In this section, we are interested in the problem of interpolation on the phase space described in the previous.
 The marginals  $[e_t]_*(\mu)$ are densities defined on the tangent space $TM$. If we only specify 
the marginals at  time $0$ and $1$  as empirical measures:  $[e_0]_*(\mu) = \sum_{i=1}^k \alpha_i  \, \delta_{x_i}  \delta_{v_i} $ and 
 $[e_1]_*(\mu) = \sum_{j=1}^k \beta_j  \, \delta_{y_j}  \delta_{w_j} $, as explained in Section \ref{SecPhaseSpace}, we can simplify the Kantorovich using the exact $L^2$ norm of the acceleration of the spline between 
 $(x_v)$ and $(y,w$), whose cost is given in Formula \eqref{EqCostTM}.  
Again, let us underline that this cost is \emph{not} a Riemannian cost on the tangent space of $\R^d$ since if $v = w$ and $x,y$ are close, the cost is dominated by the term $4(|v|^2 + |w|^2 + \langle v,w \rangle)$ which need not be zero. Then, the Kantorovich problem reduces to the minimization of 
\begin{equation}
\sum_{i,j = 1}^{k,l} \pi_{i,j} c((x_i,v_i),(y_j,w_j)) \,,
\end{equation}
under the constraints
\begin{equation}
\begin{cases}
\sum_{i = 1}^k \pi_{i,j} = \beta_j \,\\
\sum_{j = 1}^l \pi_{i,j} = \alpha_i \,.
\end{cases}
\end{equation}
It is straightforward to apply entropic regularization/Sinkhorn in this case 
which amounts to add, for a positive parameter $\varepsilon$, $\varepsilon \sum_{i,j} \pi_{i,j} \log(\pi_{i,j})$ to the previous linear functional and to numerically solve the corresponding variational problem with the Sinkhorn algorithm \cite{Sinkhorn67,Cut}  (See also appendix \ref{App2} where Sinkhorn algorithm is detailed in the more general multimarginal case). 
It is interesting to note that the choice of $\varepsilon$ is more delicate than in the standard case of a quadratic distance cost.  \\ 

In Figure \ref{FigHermite}, we present the convergence rate of this  method with respect to two different values of $\varepsilon$ and the most likely deterministic plan given the optimal plan $\pi^\varepsilon$.  Note that this entropic regularization method scales with the number of points as $N^2$ and is valid in every dimension.
\begin{figure}
 \centering	
  \begin{tabular}[h]{cc}
\includegraphics[width=8cm]{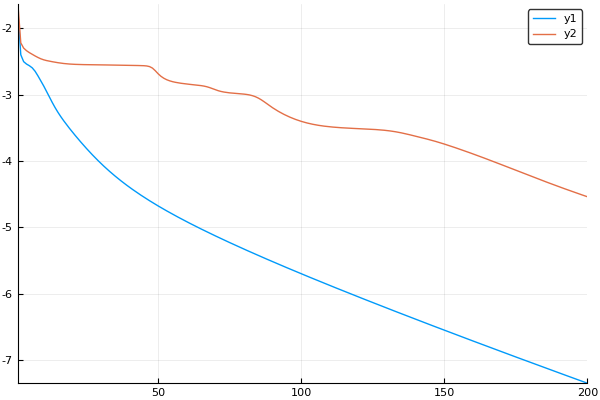}
&
\includegraphics[width=8cm]{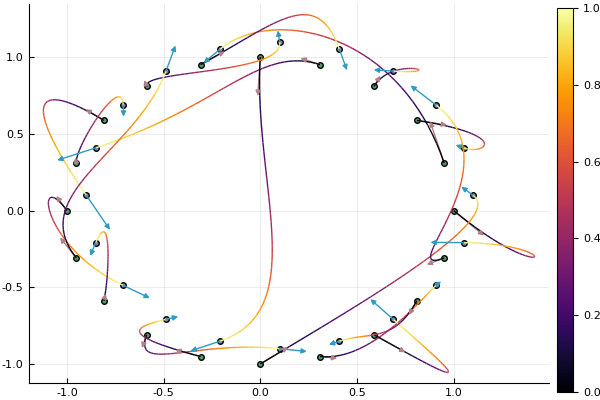}
\end{tabular}
\caption{Convergence (left) and Hermite interpolation problem between Two empirical measure in phase space (right).  We represent the most likely splines in the position space.  }
\label{FigHermite}
\end{figure}
\par

\subsection{MultiMarginal  formulation}  

This is the direct discretization of \eqref{EqMultiMarginal}  which avoids working in phase space with the cost (\ref{psc}) thus enabling  fast computations in 2D.  
 In what follows,  the time cylinder $[0,1] \times M $ is  discretized in time as $ \bigotimes_{i=0,N} M_i$, the product space of $N+1$ copies of $M$ at each of the 
 $N+1$ time steps.  We will use  a regular time step discretization  $\tau_i = i\, d\tau$ where $d\tau =\frac{1}{N}$.
 Using a  classic finite difference approach, the time discretization of  (\ref{EqMultiMarginal}) is 
  
\begin{equation}
\label{JN} 
\min_{\mu_{d\tau} } \int_{\bigotimes_{i=0,N} M_i } c_{d\tau} (x_1,...,x_N)  \ud \mu_{d\tau}(x_1,..,x_N) \,,
\end{equation}
where $\mu_{d\tau}$ now spans the space of probability  measures on  $\bigotimes_{i=0,N} M_i$ representing\
the space of piecewise linear curves passing through $x_0,x_1,...,x_N$ at times $\tau_0, ..., \tau_N$. \\

A straighforward computation gives 
 \begin{equation}
 \label{acost} 
 c_{d\tau} (x_1,...,x_N)  :=   \sum_{i=1,N-1}   \dfrac{ \| x_{i+1}+x_{i-1}-2\, x_i \|^2 }{d\tau^3} 
\end{equation}

For all  times,  marginals (\ref{TimeMarginals}) are computed as  : 
\begin{equation}
\label{TM} 
\tau_j  \mapsto \int_{\bigotimes_{i \neq j } M_i }  \ud \mu_{d\tau} (x_1,..x_N)  
\end{equation}

In order to simplify the presentation we will assume that the marginal constraints (\ref{EqMarginalConstraints})
are set at times $t_1,..t_n$ which coincide with times steps of the discretization (of course $n < N$, meaning 
the number of constraint is not the same as the number of time steps). 

In short,  there exist $(j_1,..j_n) \in [0,N]$ such that 
\[
(t_1,..,t_n) =  (\tau_{j_1}, ... , \tau_{j_n}). 
\]
 The constraint  (\ref{EqMarginalConstraints}) becomes  for all $k = 1,..n$ 
\begin{equation}
\label{JC} 
\int_{\bigotimes_{i \neq j_k } M_i }  \ud \mu_{d\tau} (x_1,..x_N)   = \rho_{j_k} ( x_{j_k}) 
\end{equation}
where $\rho_{j_k}$ is the prescribed density to interpolate at time $\tau_{j_k} = t_k$.  \\

The time discretized problem is the multimarginal problem (\ref{JN} -\ref{JC}). \\

The simplest space discretization strategy is to use a regular  cartesian grid.  In dimension 2 and for $M = [0,1]^2$ and at time $t_i$, the grid 
will be denoted  $x_{\alpha_i,\beta_i} = (\alpha_i \, h,\beta_i h)$  for $(\alpha_i,\beta_i) \in [0,N_x]$ and $h = \frac{1}{N_x}$, $ a = \{\alpha_i\}$ 
and $b = \{\beta_i \}$ will be the vectors of indices.

The time and space discretization of the problem then becomes 

\begin{equation}
\label{JND} 
\min_{ T}  \sum_{a,b}   C_{a,b}  \, T _{a,b} 
\end{equation}
Where $T$ is the  $N\times N_x \times N_x$ tensor of grid values  $\mu_{d\tau}(x_{\alpha_1,\beta_1} ,..,x_{\alpha_N,\beta_N})$ and 
\begin{equation}
\label{acostD} 
C_{a,b} =  c_{d\tau} (x_{\alpha_1,\beta_1} ,..,x_{\alpha_N,\beta_N}) 
\end{equation} 
The marginals (\ref{TM}) at all times $\tau_j$ are given by 
\begin{equation}
\label{TMD} 
 \sum_{a \setminus \{\alpha_{j} \}  , \, b \setminus \{\beta_{j} \}  }  T _{a,b}  
\end{equation}
The constraints  (\ref{JC}) therfore becomes for all $k$ 
\begin{equation}
\label{JCDk} 
 \sum_{a \setminus \{\alpha_{j_k} \}  , \, b \setminus \{\beta_{j_k} \}  }  T _{a,b}   =  \rho_{j_k} ( x_{\alpha_{j_k}, \beta_{j_k}} ) 
\end{equation}
$a \setminus \{\alpha_{j_k} \}  $ denotes the set of indices $a$ minus $\alpha_{j_k}$. \\

 The Entropic  regularized  problem is 
\begin{equation}
\label{JNDe2} 
\min_{ T^\epsilon}  \sum_{a,b}   \{  C_{a,b}  \, T^\epsilon_{a,b} + \epsilon  \, T^\epsilon_{a,b} \,  \log(T^\epsilon_{a,b}) \} 
\end{equation}
and easier to solve. See Appendix \ref{App2} for a description of Sinkhorn algorithm.

\subsection*{Numerical Simulations} 

\paragraph{\textbf{1D case: }}We present, figures \ref{FigBasicExample} and  \ref{FigBasicExample2},   a 1D test case to highlight some of the qualitative properties of the cubic splines interpolation on the space of densities. 

\begin{figure}[htb]
\begin{center}
\begin{tabular}{cc}
\begin{tabular}{c}
$\null$\hspace{-0.35cm}
\includegraphics[width= 6cm]{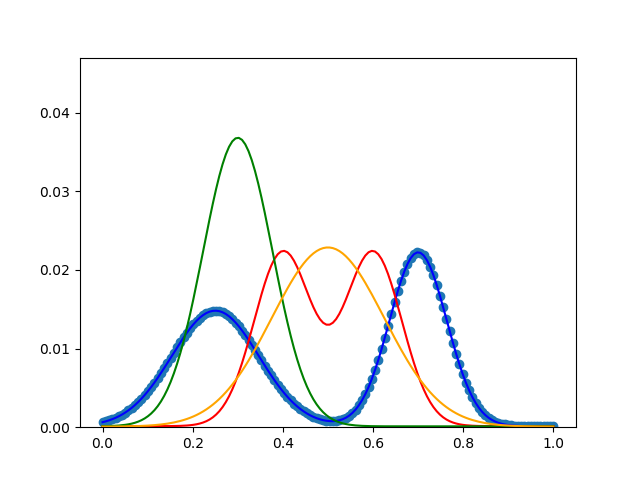}
\\
Initial time data and targets
\end{tabular}
\\
\begin{tabular}{ccc}
$\null$\hspace{-0.35cm} \includegraphics[width= 5.1cm]{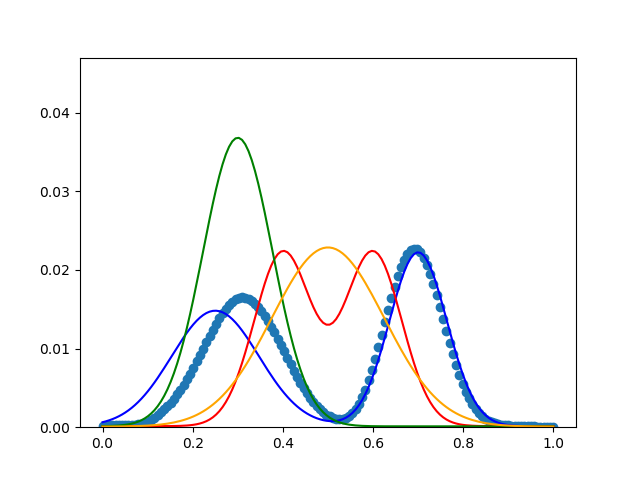}&
$\null$\hspace{-0.45cm} \includegraphics[width= 5.1cm]{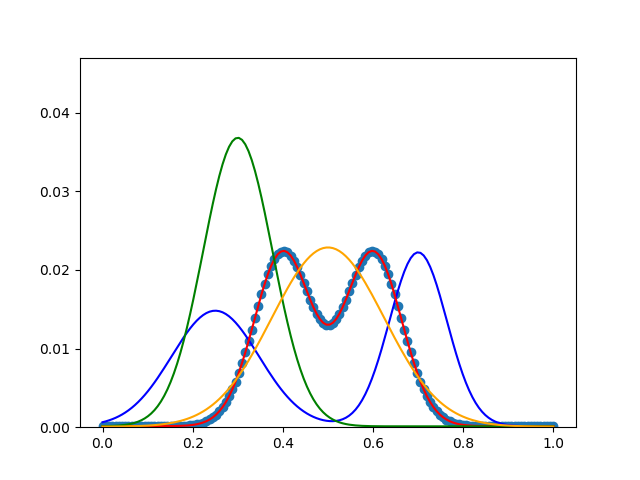}&
$\null$\hspace{-0.45cm} \includegraphics[width= 5.1cm]{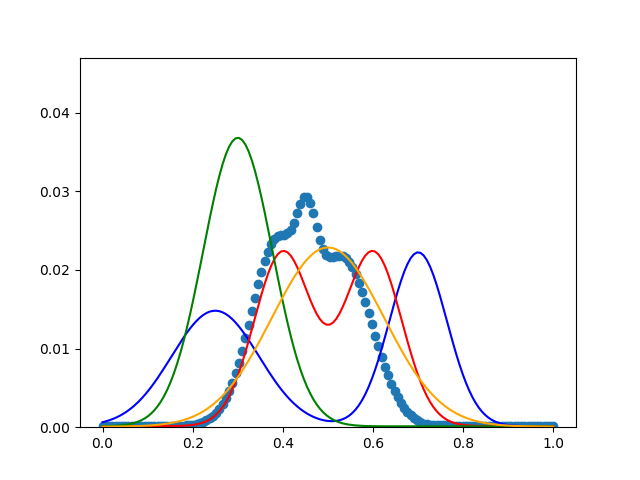}\\ 
$\null$\hspace{-0.35cm} \includegraphics[width= 5.1cm]{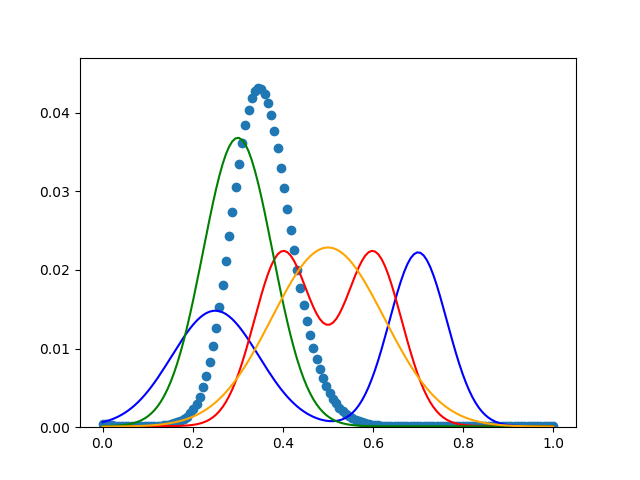}&
$\null$\hspace{-0.45cm} \includegraphics[width= 5.1cm]{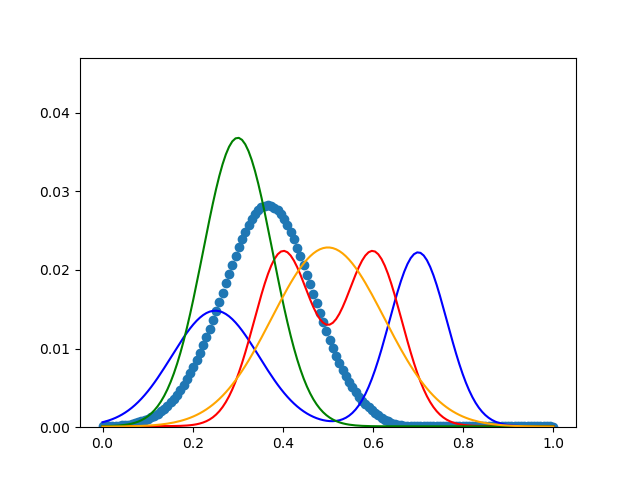}&
$\null$\hspace{-0.45cm} \includegraphics[width= 5.1cm]{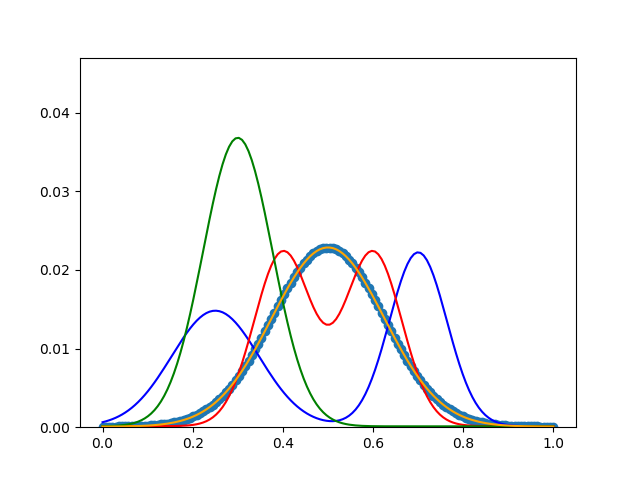}
\end{tabular}
\end{tabular}
\caption{Four interpolation timepoints, $1,6,11,16$ and representation of the four density configurations, as well as $6$ intermediate times. 
The doted line represent the reconstructed density curve in time.
This experiment underlines that the spline curve has more smoothness in time and can present some concentration or diffusion effects depending on the data which would not be present for the usual Wassertein geodesic. The entropic regularization parameter is $\varepsilon = 8.10^{-5}$.
}
\label{FigBasicExample}
\end{center}
\end{figure}
We consider four interpolation time points and the corresponding data are mixture of Gaussians of different standard deviations. We use a discretization of $140$ points on the interval $[0,1]$ with $16$ time steps. The doted line represent the reconstructed density curve in time.  This experiment shows that the mass can concentrate or diffuse in some situation.  

Another  important point here is that the entropic regularization parameter has an important impact on this concentration/diffusion effects: we show the simulations for $\varepsilon = 0.002$ and $\varepsilon = 8.10^{-5}$. In the simulation with a large $\varepsilon$, the concentration effect is not present and it is due to the diffusion on the path space.

\begin{figure}[htb]
\begin{tabular}{ccc}
$\null$\hspace{-0.35cm} \includegraphics[width= 5.1cm]{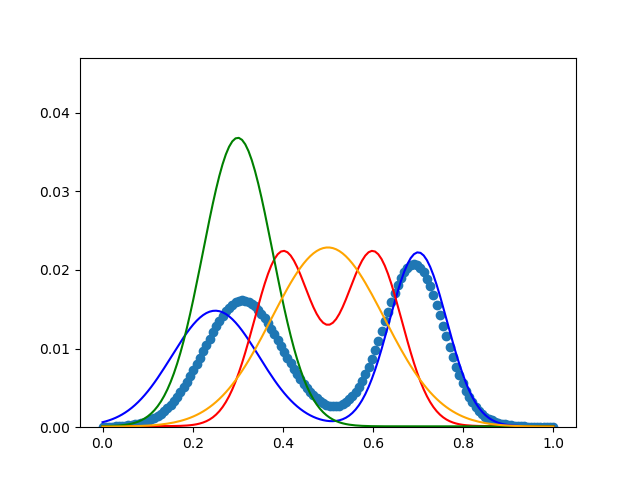}&
$\null$\hspace{-0.45cm} \includegraphics[width= 5.1cm]{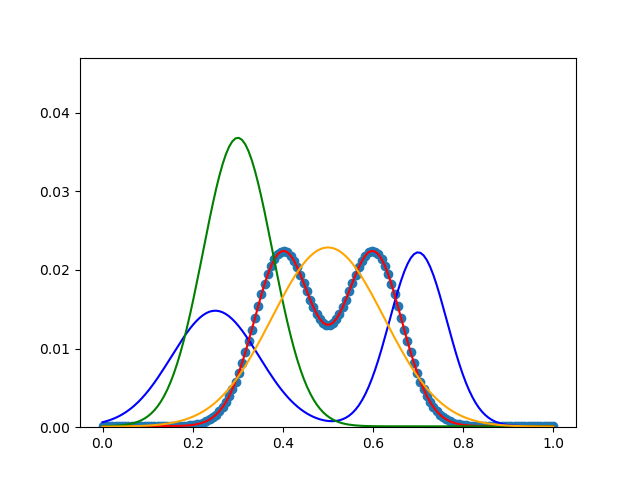}&
$\null$\hspace{-0.45cm} \includegraphics[width= 5.1cm]{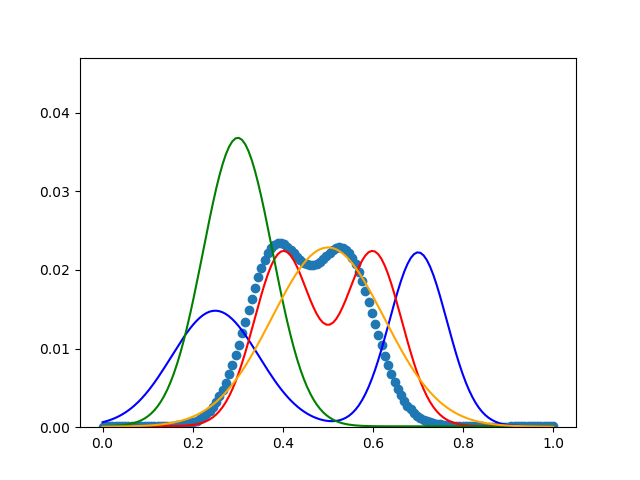}\\ 
$\null$\hspace{-0.35cm} \includegraphics[width= 5.1cm]{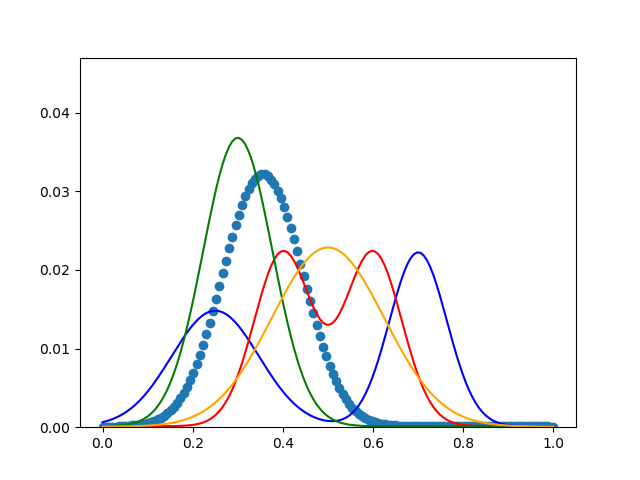}&
$\null$\hspace{-0.45cm} \includegraphics[width= 5.1cm]{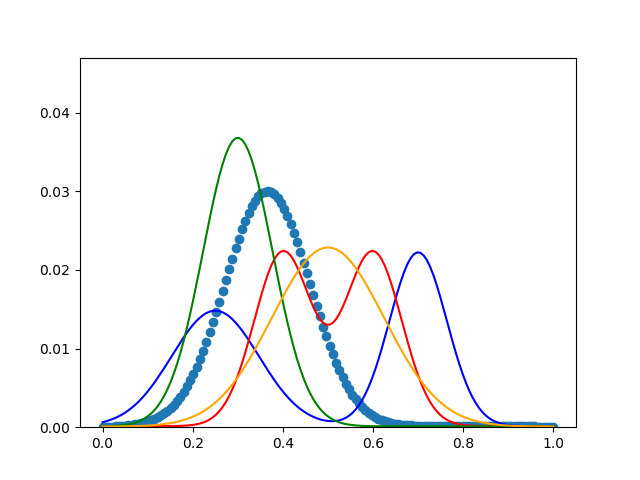}&
$\null$\hspace{-0.45cm} \includegraphics[width= 5.1cm]{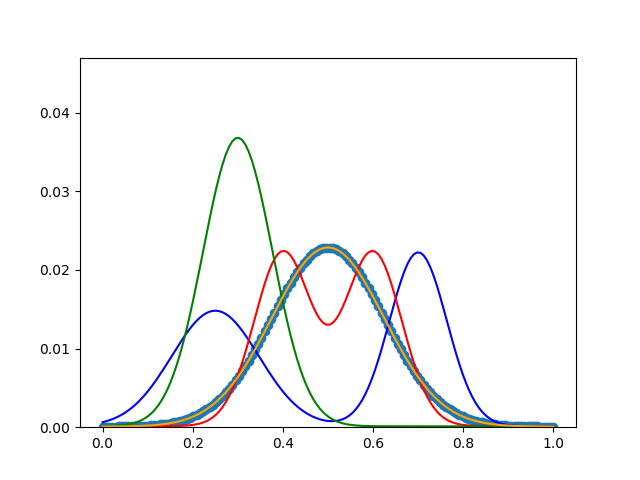}
\end{tabular}
\caption{The same experiment with a larger entropic regularization parameter $\varepsilon = 0.002$. As expected, we observe less concentration of mass.
}
\label{FigBasicExample2}
\end{figure}

\paragraph{\textbf{2D case: }}We present a 2D test case which computes a Wasserstein spline in the sense of (\ref{JND}) interpolating four Gaussian identical 
densities at time 1, 5, 13, and   17, see figure \ref{f2d1}. We use a time step $d\tau =1$  and 17 $N=17$ time steps.  The space discretization is $Nx = 50$. 
The entropic regularization parameter is $\epsilon = 0.002$, note that the stability of the method depends on this parameter. 
It also generates artificial diffusion as it becomes more costly top concentrate the available mass on fewer Euclidean splines between the points of the support of the 
four Gaussians. We can compute the interpolating densities at intermediate times using (\ref{TMD}) but is more 
interesting to represent in figure \ref{f2d2} the contour line of the third  quartile, i.e. the highest values of the densities representing 1/4 of the total mass. 
Comparing with figure \ref{rotation2}, it seems clear that the Entropy diffusion spreading  pollutes the solution of the original problem (without entropic regularization).  \\

We compare this solution with the classical Quadratic cost Optimal Transport interpolation, i.e. with the speed instead of the acceleration 
in the cost.  More precisely taking : 
\begin{equation}
 \label{acost2} 
 c_{d\tau} (x_1,...,x_N)  :=   \sum_{i=0,N-1}   \dfrac{ \| x_{i+1}- x_i \|^2 }{d\tau} 
\end{equation}

As expected the mass follows respectively the linear interpolation or the Euclidean spline interpolation of the center of the Gaussians which 
are represented as thick red lines in figure \ref{f2d1}. \\

Finally we show the convergence of the Sinkhorn iterate for both simulations in figure \ref{f2d2}.
The convergence is much slower for the speed case but we did not optimize the implementation which does not need tensors and instead 
just used a degraded version of the acceleration code. This may be the reason for this strange difference.

\begin{figure}
 \centering	
  \begin{tabular}[h]{cc}
\hspace{-2cm} \includegraphics[width=10cm]{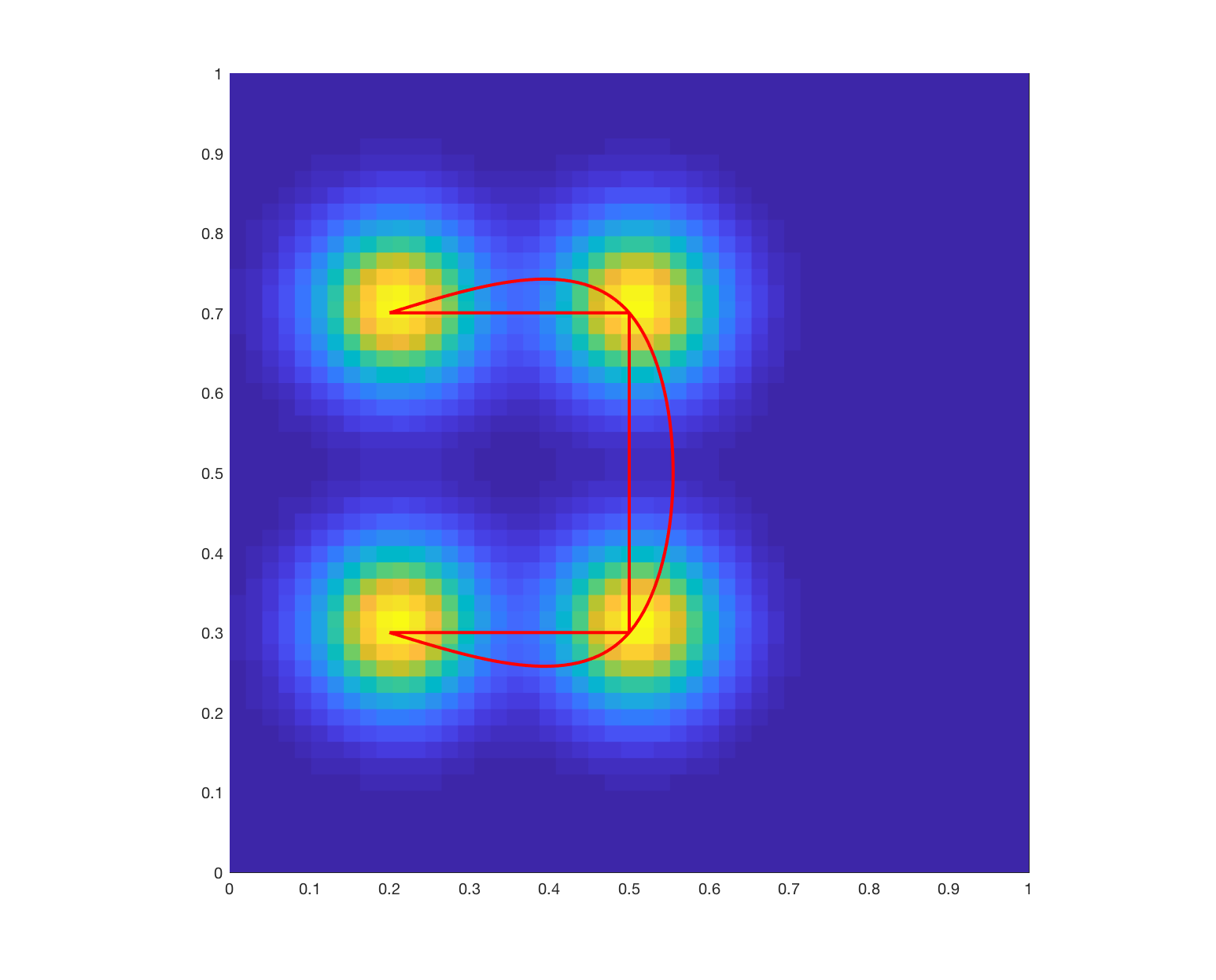}
&
\hspace{-2cm} \includegraphics[width=10cm]{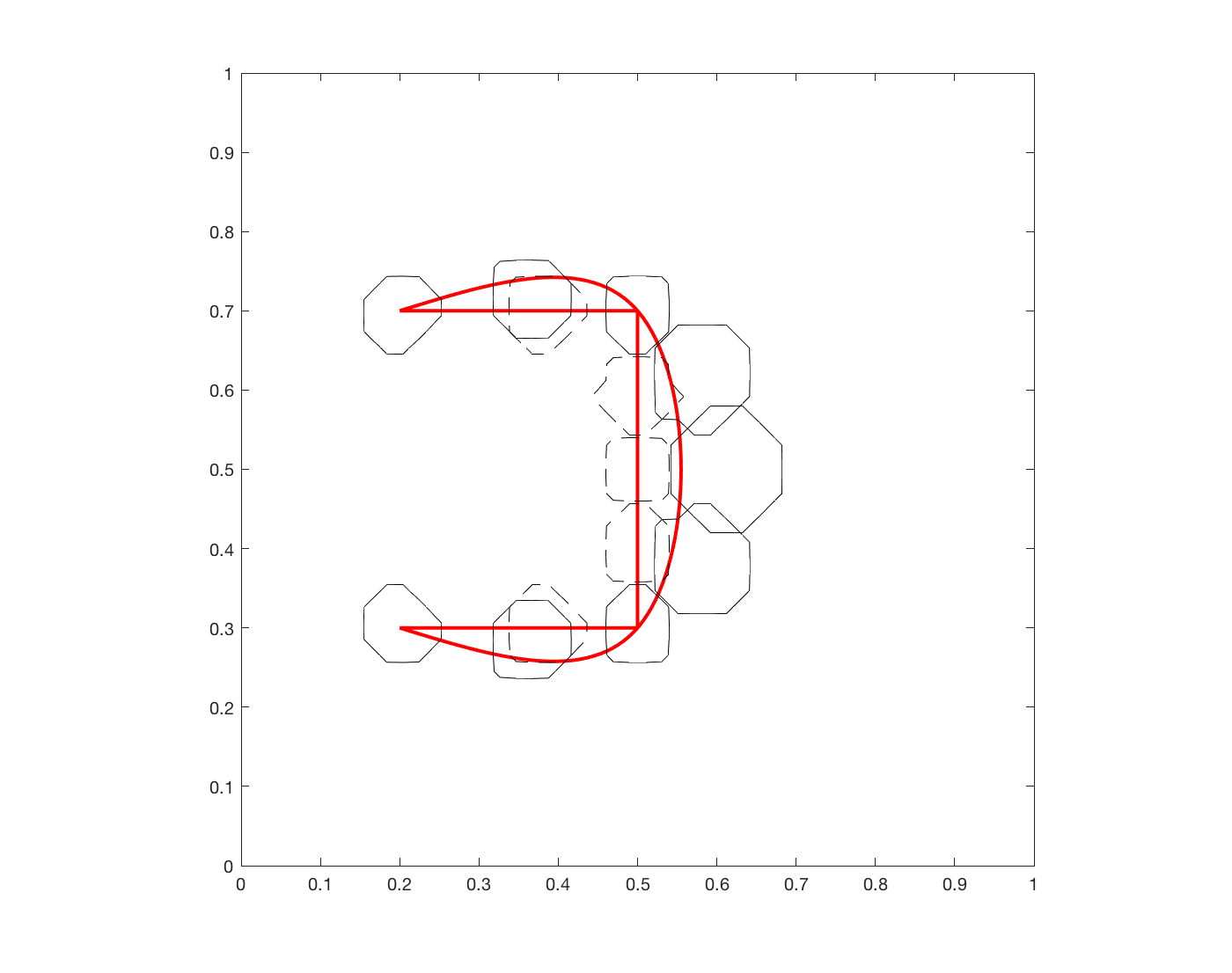}
\end{tabular}
\caption{ Spline interpolation of Four Gaussians with 17 times steps.  Left : the data and the linear and classic cubic spline  
interpolation of the of Gaussian center point. Right :  the level curve of the  third quartile of the density every 2 time steps, in solid line for our Spline Wasserstein interpolation and 
in dashed line for the classic quadratic cost (speed) interpolation.  } 
\label{f2d1}
\end{figure}

\begin{figure}
 \centering	
  \begin{tabular}[h]{cc}
\hspace{-1cm}\includegraphics[width=8cm]{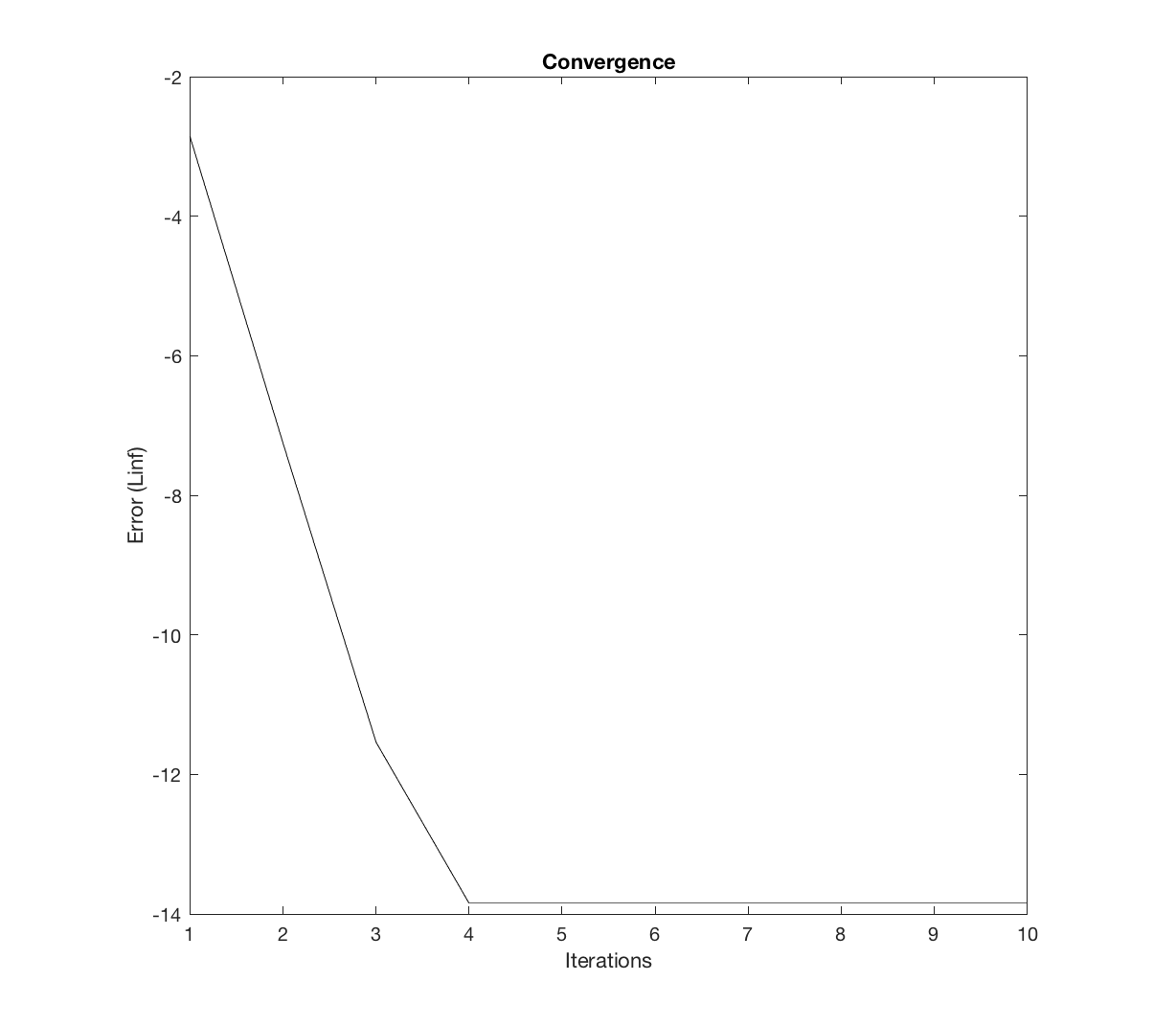}
&
\hspace{-1cm}\includegraphics[width=8cm]{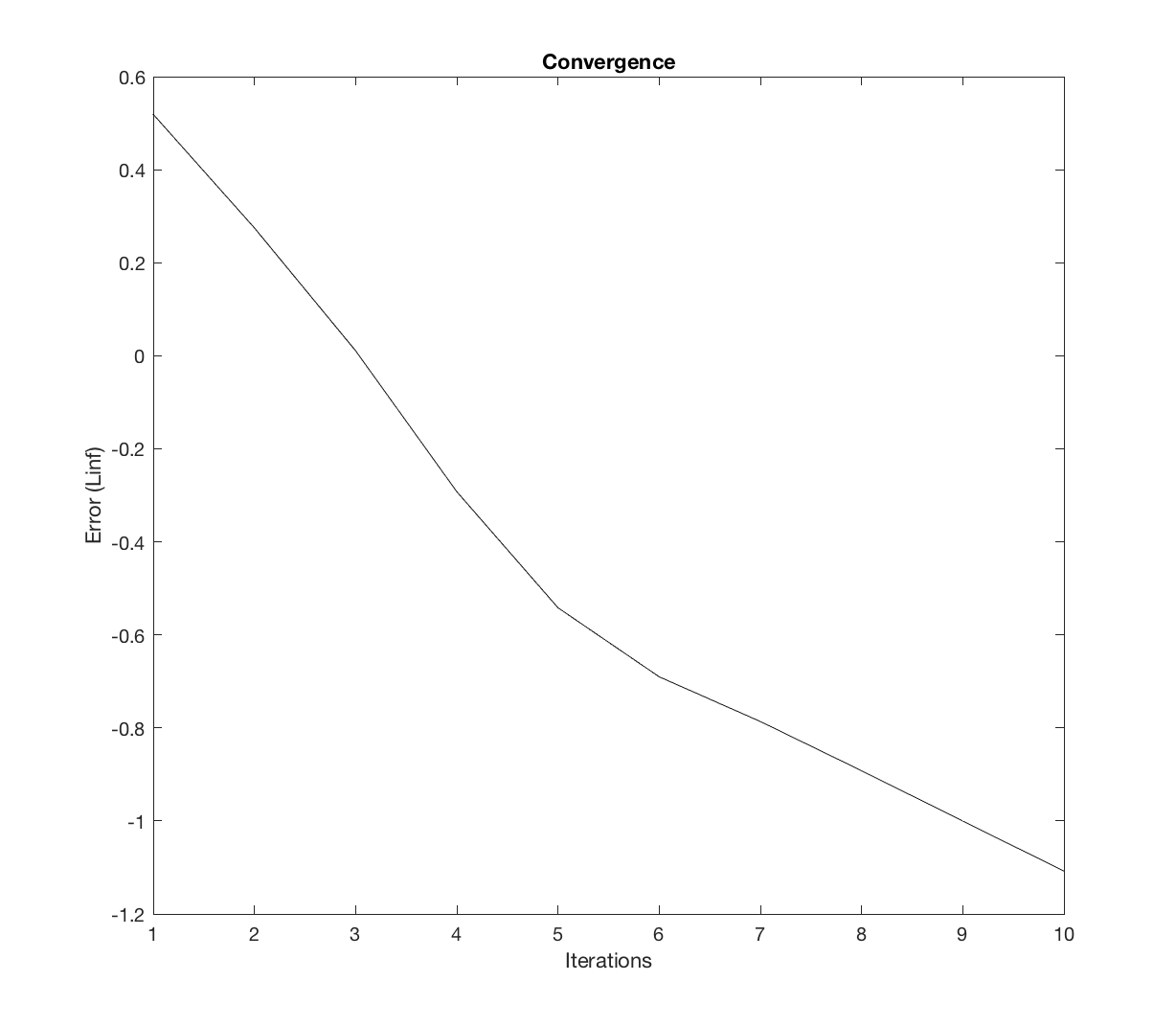}
\end{tabular}
\caption{ Convergence, i.e. Infinity norm of the difference of the Dual unknown between to Sinkhorn iteration.  This is computed every 10 iterations. 
 Left :for the acceleration cost, right :  for the speed cost . } 
\label{f2d2}
\end{figure}

\subsection{Semi-Discrete approach}\label{SD}
We propose another numerical scheme based on the semi-discrete approach introduced by M\'erigot in \cite{Mer} in dimension 2 and developed by Levy \cite{Levy} in dimension 3. Here we approximate the optimal plan $\pi$ in the formulation \eqref{eqK} by a sum of N tensor product of diracs masses. That is $\pi_N =\sum^N_{j=1}  \left( \bigotimes^n_{i=1}  \frac1N  \delta_{X^i_j}\right)= \sum^N_{j=1}   \frac1N  \delta_{\left( X^1_j,\ldots,X^n_j\right)} $. 
\begin{remark}
Since there is a unique corresponds between $n$ points $\left( X^1_j,\ldots,X^n_j\right)$ and the spline $c_{X^1_j,\ldots,X^n_j }$ passing through these points at time $(t_1,\ldots.t_n)$ the measure $\pi_N $ can also be seen as $N$ direct masses defined over the set of splines: $\pi_N = \sum^N_{j=1}   \frac1N  \delta_{c_{ X^1_j,\ldots,X^n_j}}$. 
\end{remark}

 We then have to relax the constraint $(p_i)_*(\pi) = \rho_i$  since $(p_i)_*(\pi_N)= \sum^N_{j=1}   \frac1N  \delta_{X^i_j} $ cannot be absolutely continuous. 
 It leads to the following variational problem. 
\begin{definition}[Semi-discrete variational problem]\label{SDVdef}
Let $\epsilon>0$, $0 = t_1 < \ldots < t_n =1$, $n \geq 3$ and $(\rho_i)_{i=1\ldots n}$ be n absolutely continuous measures. Recall that $c(Y_1,\ldots,Y_n)$ is the cost of the cubic spline passing through the points $(Y_1,\ldots,Y_n)$  at time $(t_1,\ldots.t_n)$. 
Let 
$$\mathcal{Q}^N = \left\{ \sum^N_{j=1}   \frac1N  \delta_{\left( X^1_j,\ldots,X^n_j \right) } \middle|  (X_j)_{j=1,\ldots,N}\in M^{n} \right\}.$$ Then the semi-discrete variational problem, (SDV), is given by 
\begin{equation}\label{SDV}
(SDV) = \min_{\mathcal{Q}^N} \, \frac{1}{N}\sum_{j=1}^N c(X^1_j,\ldots, X^n_j) + \sum^n_{i=1} \frac{1}{2\epsilon^2} W_2^2\left(\sum^N_{j=1}   \frac1N  \delta_{X^i_j} , \rho_i \right),
\end{equation}
where $W_2$ is the classical Wasserstein distance given by the quadratic cost.
\end{definition}

The main drawback of this method is that, as illustrated in the numerical simulations below, the problem $(SDV)$ is not convex. 

\subsubsection{Implementation} \label{SecImplementationSemiDiscrete}
In order to solve numerically the minimization problem $(SDV)$ we use the reformulation of the spline cost in the phase space, that is in $\R^d$, with $t_{i+1}-t_i = \delta_i$:
\begin{equation}\label{equivalencedescout}
	c\left(Y_1,\ldots,Y_n\right) = \min_{\left(V_1,..V_n\right)\in (\R^d)^n} \sum_{i=1}^{n-1}  \frac{1}{\delta_i^3} 
	c_{ph}\left[\left(Y_i,\delta_i V_i\right), \left(Y_{i+1},\delta_i V_{i+1}\right)\right]
\end{equation}
where
\begin{equation}\label{EqCostTM2}
c_{ph}[(x,v),(y,w)] = 12|x-y|^2 + 4(|v|^2 + |w|^2 + \langle v,w \rangle + 3 \langle v+w, x-y \rangle).
\end{equation}
The advantage of the formulation \eqref{equivalencedescout} is that the cost is separable in the phase space and the gradient with respect to speeds and positions is easy to compute. 
 
 We thus implement a gradient descent in the phase space using the lbfgs function in python. We compute the gradient by automatic differentiation. 
 The Wasserstein terms in the minimization problem \eqref{SDV} depends only on the positions and are computed thanks to M\'erigot Library \cite{merigotgit} in dimension 2.
  To do simulations in dimension 3 one has to use L\'evy Library \cite{levywebpage}. The density constraints $\rho_i$ are given trough linear functions on a triangulation.
 
\begin{remark}
Other problems can be addressed using similar optimization problem as in Definition \ref{SDVdef}. 
 For instance the quadratic cost in \eqref{SDV} leads to Wasserstein interpolation. 
 We can also interpolate with curves as smooth as we want, using for instance the $L^2$ norm of the derivative of order $m$ of the curve or even other classical interpolating curves. 
 \end{remark}

  
 \subsubsection{Numerical simulations}
  
  We propose three numerical simulations, one to compare the qualitative results with respect to the multi marginal approach and especially Figure \ref{f2d1}. 
  A second one in order to illustrate the non-convexity issue and a third one for applications in images.

\paragraph{\textbf{The rotation case: Figure \ref{rotation2}.}}
In this case we compute Wasserstein splines passing through four gaussians with variance 15 and center of masses respectively $(0,2),(10,0),(10,6),(0,4)$ with constraint parameter $\epsilon= 10^{-3}$. The number of points is $2000$. 
In this case the result is a global minimizer and is not sensible to initialization. 
The lack of convexity is not an issue.
Compare to Figure \ref{f2d1}, this approach gives a better a approximation of the intermediate densities especially with less diffusion.

\begin{figure} \centering	
\begin{tabular}[h]{cc}
\hspace{-2cm} \includegraphics[width=10cm]{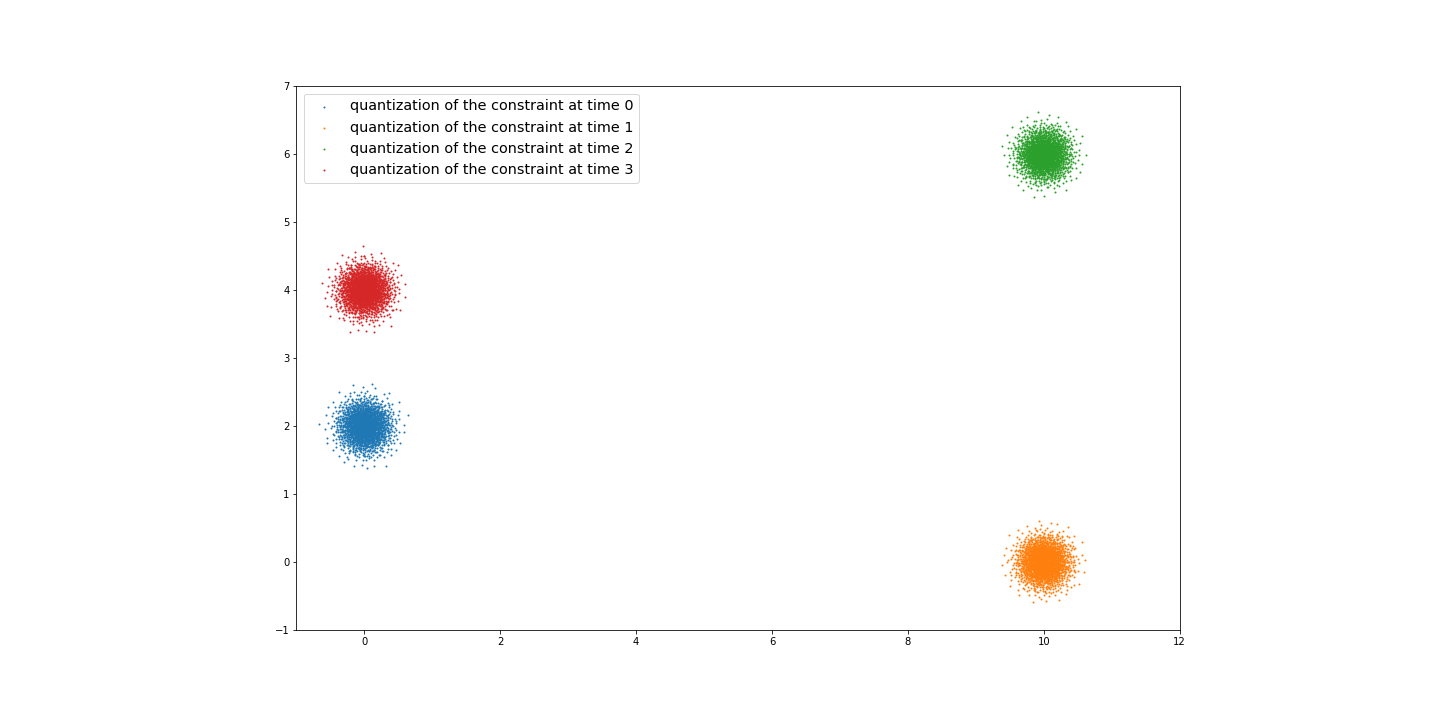}
\hspace{-2cm} \includegraphics[width=10cm]{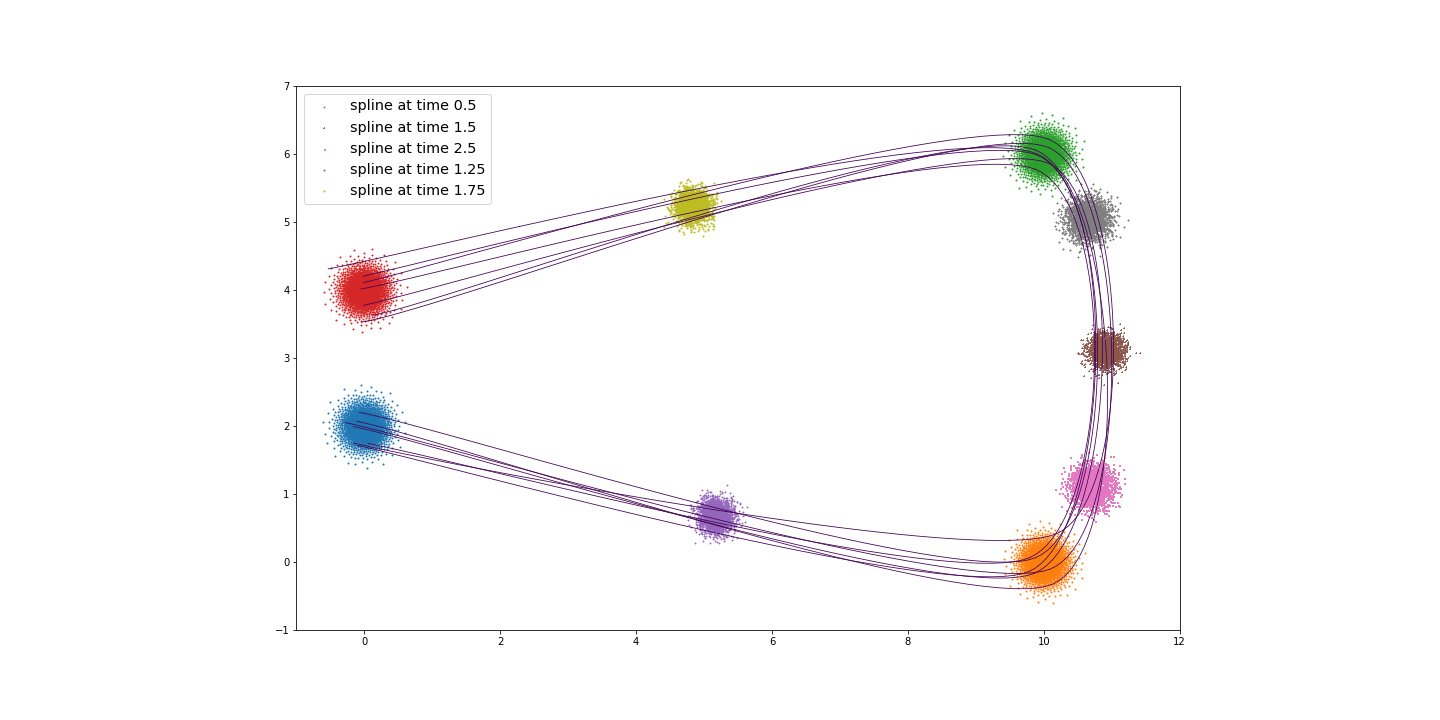}  \\
\hspace{-2cm} \includegraphics[width=15cm]{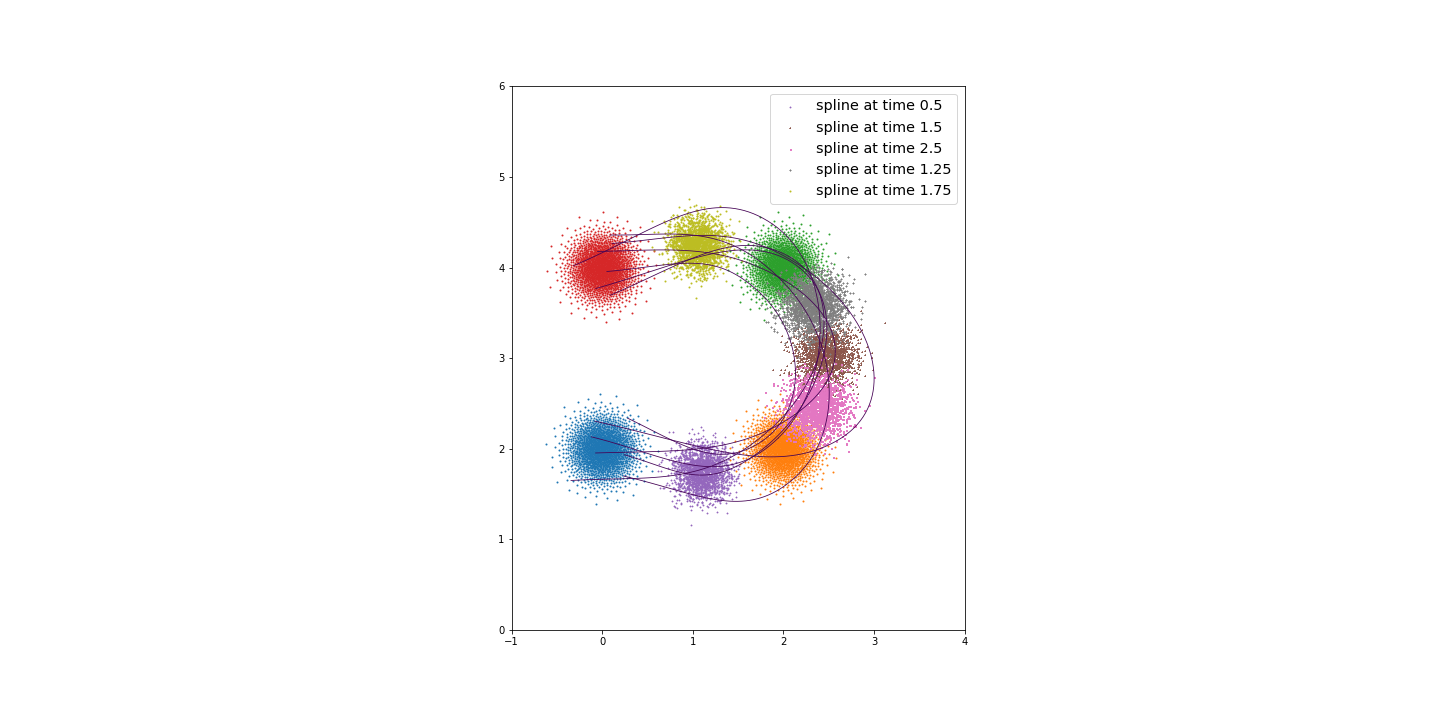}

\end{tabular}
\caption{Spline interpolation for gaussians with 2000 Dirac masses for each measure, $\epsilon=1O^{-3}$. 
Left: sample of each density constraints $\rho_i$, i=1,2,3,4.
Right: Some trajectory of diracs masses randomly chosen, marginals at the constrained time $0,1,2,3$ and marginals at time $0.5,1.2,1.5,1.7,2.5$.    
Second Line : the same configuration as in figure \ref{f2d1}.
} 
\label{rotation2}
\end{figure}

\paragraph{\textbf{The crossing case: Figure \ref{crossing0}, \ref{crossing} }}
  Here we compute Wasserstein splines starting from a mixture of two gaussians with centrer $(0,-1),(0,1)$ and variance $15$ then passing through a gaussian with center $(0,0)$ and variance $15$ and finishing at a translation of the initial mixture. The number of points is $2000$, $\epsilon$ will value $1$ or $1000$.
   
 We expect the global minimizer to be straight lines crossing around the middle constraint and with a low cost. 
Numerically depending on the initial conditions, we can recover different local minimizers, 
the local minimum which is reached is extremely correlated with the initial coupling. 
In Figure \ref{crossing0} we observe that changing $\epsilon$ but keeping a similar initial coupling, all points are given by a quantization of the middle density with a random enumeration and $0$ initial speed, yields to a similar local minimum.

\begin{figure} \centering	
 \begin{tabular}[h]{cc}
 \hspace{-2cm} \includegraphics[width=9cm]{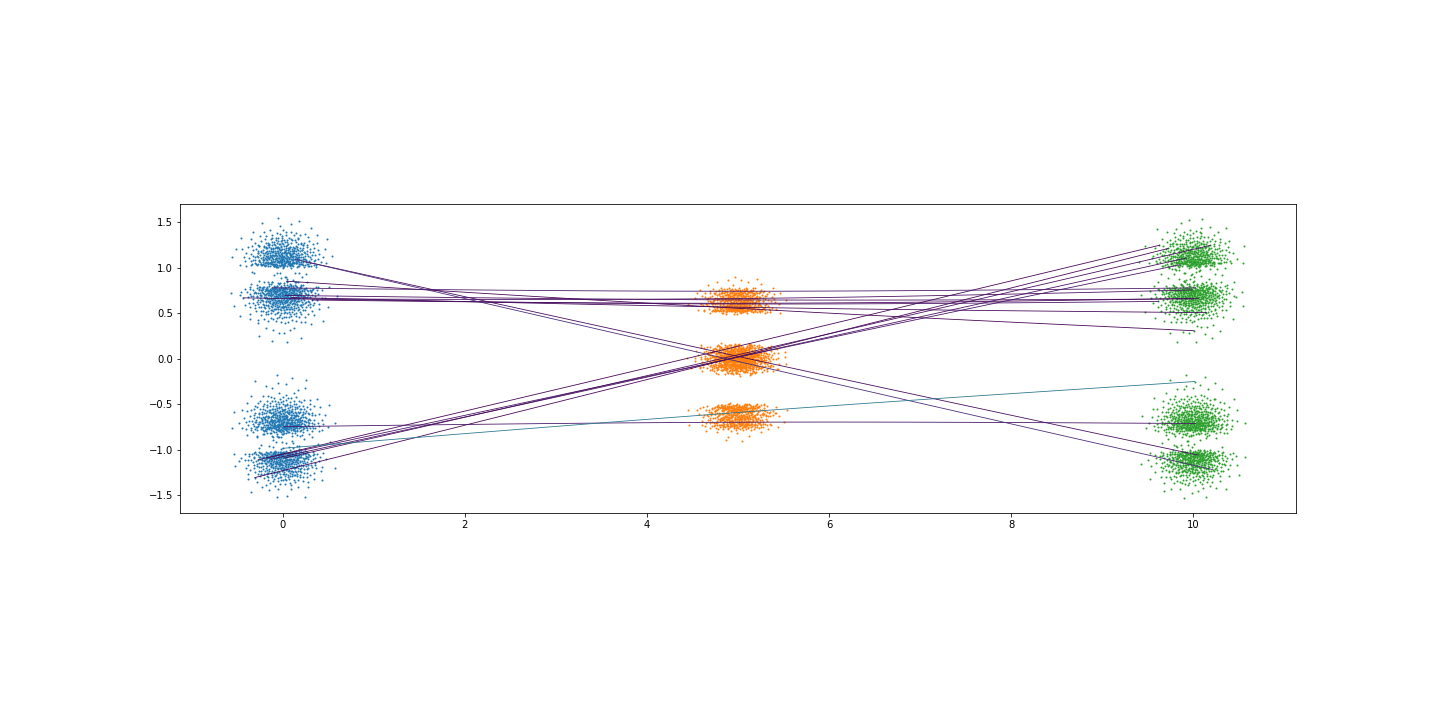} &
\hspace{-0.5cm}\includegraphics[width=9cm]{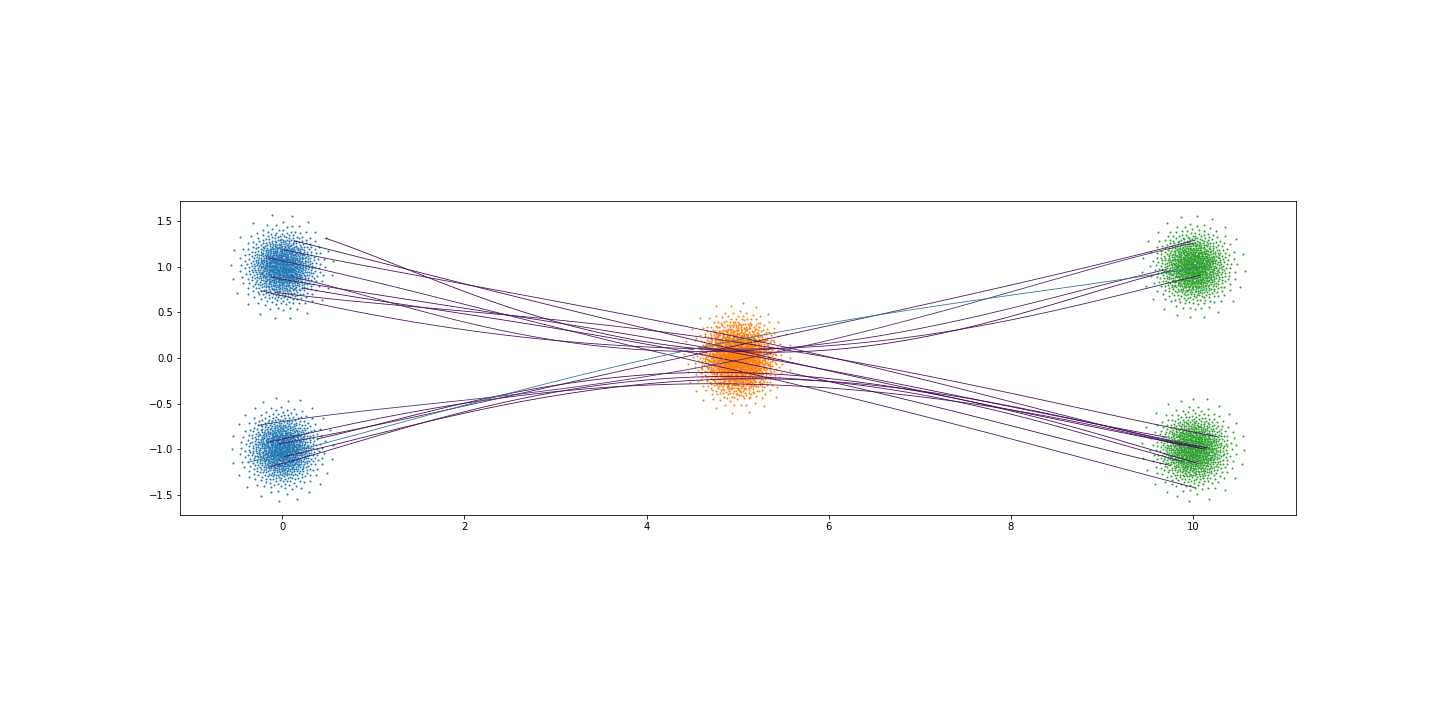} 
\end{tabular}

%

\caption{Spline interpolation for a mixture of gaussians with 2000 Dirac masses. Same initial coupling for both figure. 
Left: $\epsilon=1$.
Right: $\epsilon=1000$.   
} 
\label{crossing0}
\end{figure}

Finding a good initial coupling is the hard part in order reach the global maximum.
One solution is to initialize with points close to each other and a very large $\epsilon$. 
Then one as to add some noise in the gradient and decreases slowly $\epsilon$.
 Unfortunately we didn't find a systematic approach for this random multi-scale method and one as to fit the parameters case by case.  
In Figure \ref{crossing} the global minimizer is achieved by first computing the spline with a relaxed constraint, i.e. large $\epsilon$, only for the final time ( in pratice $\epsilon =[1000,1000,1]$. 
Then we use this result, which has the good initial coupling, as and initial condition and set $\epsilon=1000$ for all the constraints.
 We also compare this results with the interpolation with a different initial condition and the Wasserstein geodesics.
  In all these simulations we clearly observe that particles can cross along the dynamic appart from the optimal transport inthis situation.

\begin{figure} \centering	
 \begin{tabular}{cc}
 \hspace{-2cm} \includegraphics[width=10cm]{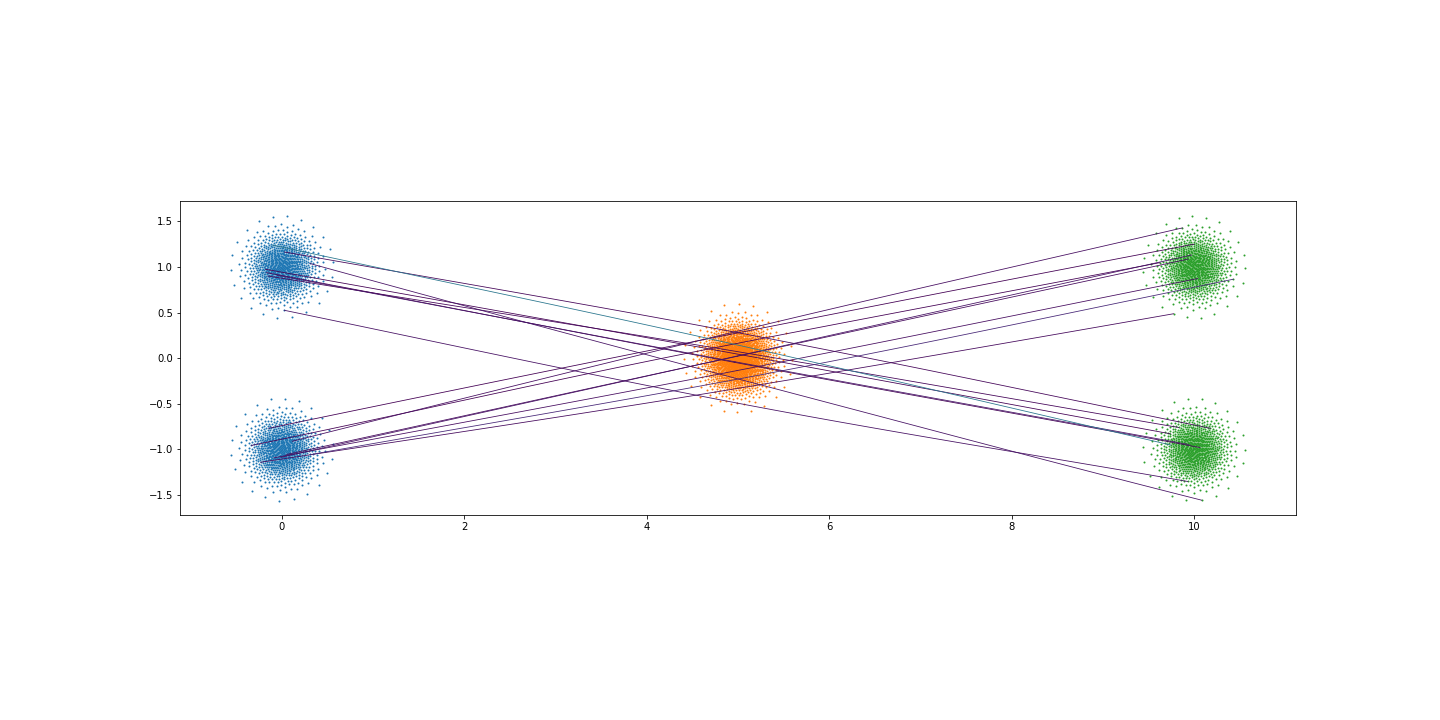} 
\hspace{-0.5cm} \includegraphics[width=10cm]{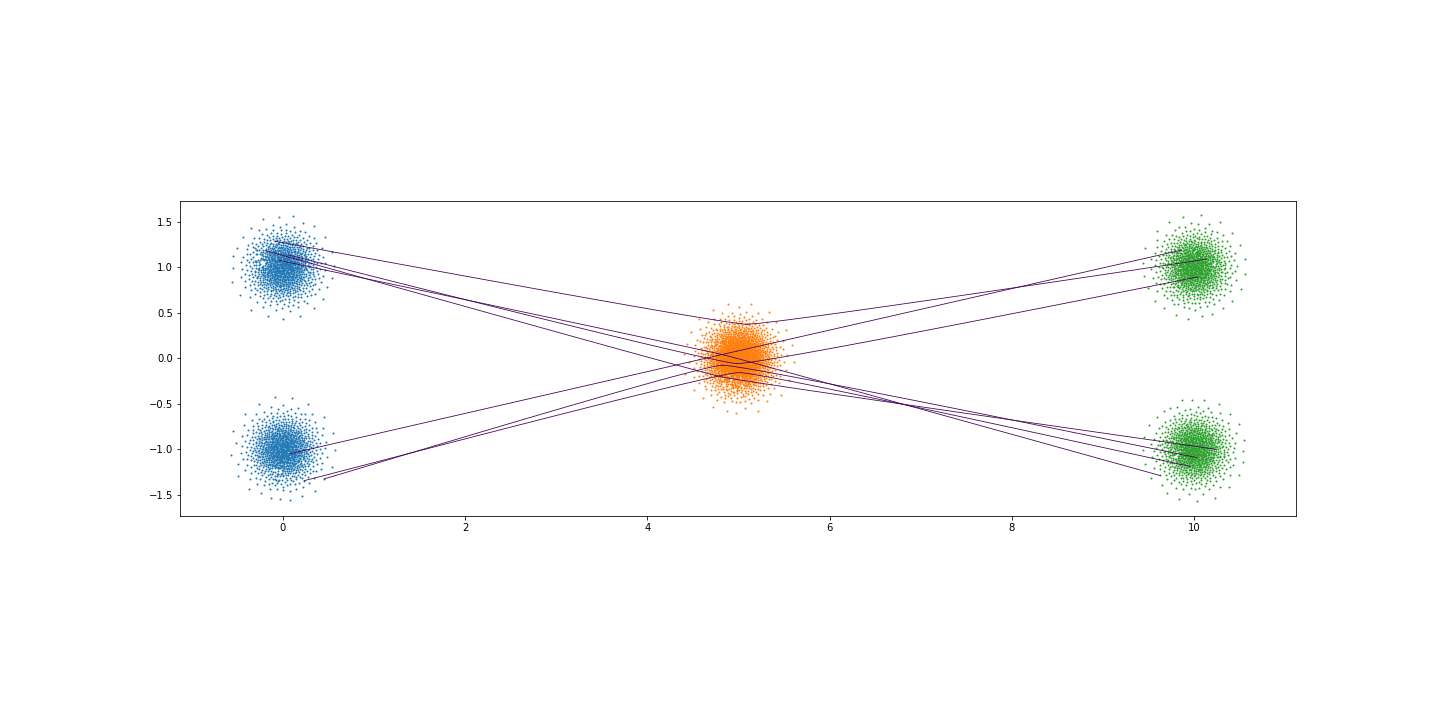} \\
\includegraphics[width=10cm]{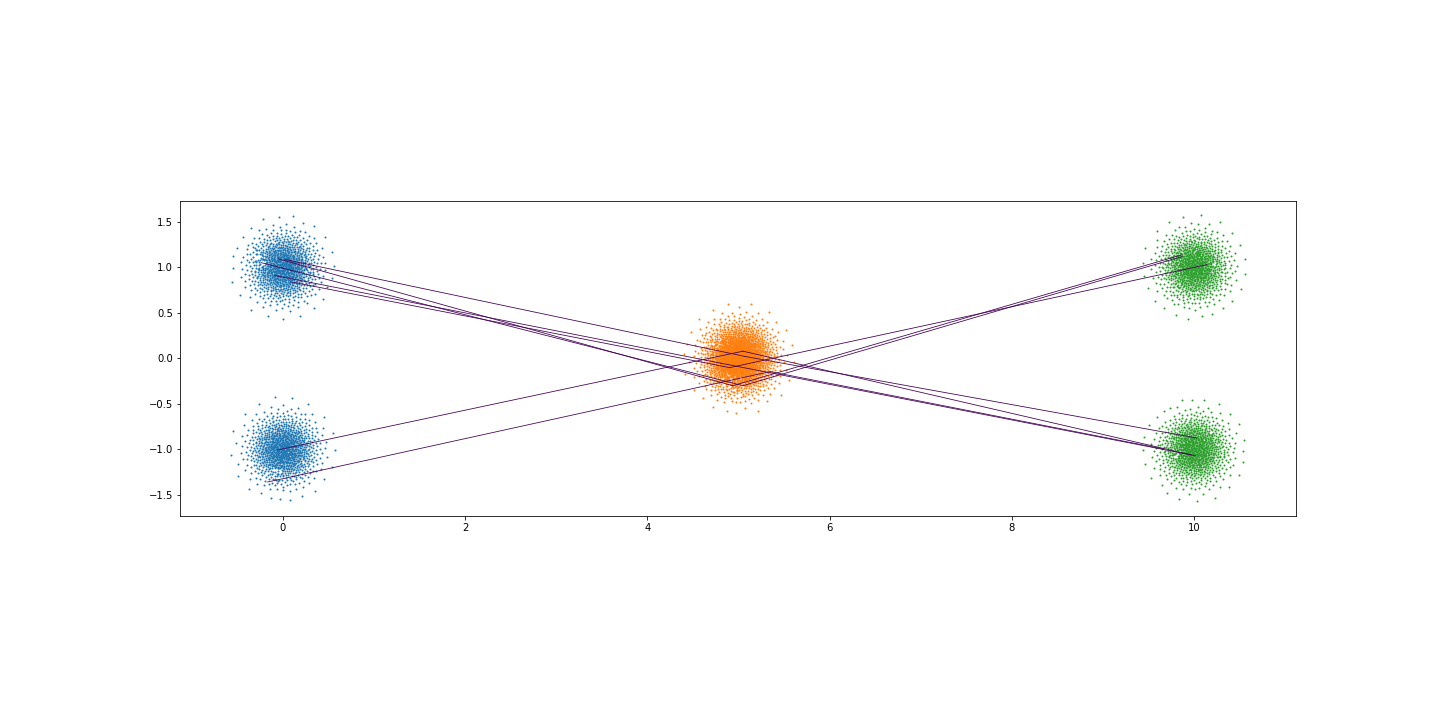} 

\end{tabular}

%

\caption{Spline interpolation for a mixture of gaussians with 2000 Dirac masses for each measure. $\epsilon=1000$.
Top Left: Initialization with a good coupling, total cost $=302$.
Top Right: Initialization with a quantization of the middle density and no speed, total cost $=804$ (local minima).
Bottom: Interpolation with the Wasserstein geodesic. $\epsilon=1000$, cost $=930$.
} 
\label{crossing}
\end{figure}  

Note that this spline approach is related to the problem of finding minimal geodesics along volume preserving maps done by M\'erigot and Mirebeau \cite{merigot2015minimal} : in their work the constraints $\rho_i$ are the Lebesgue measure, the cost is changed by the quadratic cost between two points and they have a coupling constraint.
 Therefore their minimization problem is also non convex but the coupling is given as a constraint so the non convexity issue didn't rise as clearly as in this spline problem.

\paragraph{\textbf{Image interpolation: }}
pour l'instant c'est pas presentable, ca passe vraiment au milieu. 
Je vais relancer dans la semaine mais je propose de faire une version sans.

\begin{remark}[Extrapolation]
The minimization of the acceleration can be used to provide time extrapolation of Wasserstein geodesic in a natural way: particles follow straight lines. This can be implemented in a $3$-marginal problem with the acceleration cost $c(x_1,x_2,x_3) = \frac{1}{\lambda^2}|x_3 - 2x_2 + x_1|^2 + \frac{1}{\lambda}|x_2 - x_1|^2$ under marginal constraints at time $1$ and $2$.
Note that, in the spline model, the formulation we proposed does not prevent particles from crossing each other. They are completely independent. Therefore, the particles following simply geodesic lines and after a shock, the evolution is not geodesic in the Wasserstein sense (since shocks do not occur but at initial and final times). The implementation of time extrapolation using entropic regularization is straightforward. Figures \ref{Extrapolations1} and \ref{Extrapolations2} show some experiments on $[0,1]$ discretized with $100$ points and $\varepsilon = 0.015$. The translation experiment recovers what is expected however the effect of the diffusion can be seen with a twice larger $\varepsilon$.
\begin{figure}
 \centering	
  \begin{tabular}[h]{cc}
\hspace{0cm}\includegraphics[width=8cm]{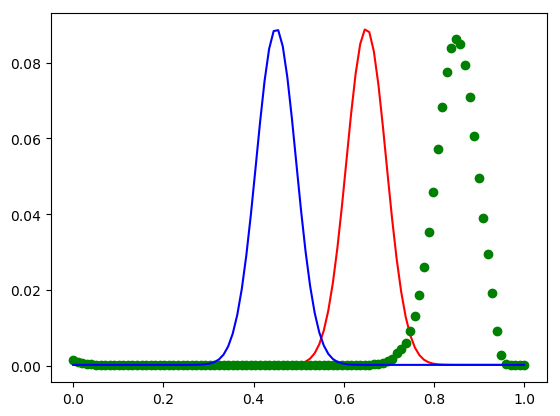}
&
\hspace{0cm}\includegraphics[width=8cm]{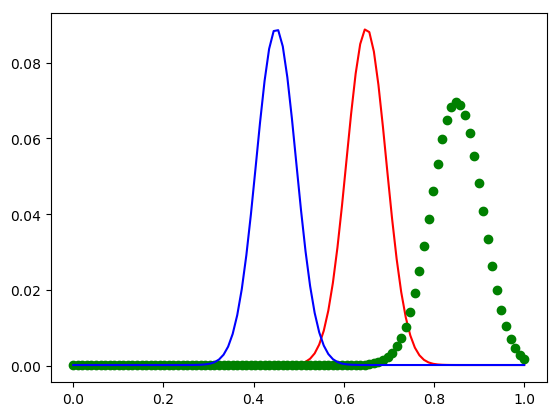}
\end{tabular}
\caption{Extrapolation of a translation with two different $\varepsilon = 0.015$ and $\varepsilon = 0.03$} 
\label{Extrapolations1}
\end{figure}
We also show two other simulations, one is a splitting simulation and the last one is a merging of two "bumps" into a single one. The extrapolation shows an other bimodal distribution which is explained by particle crossings.
\begin{figure}
 \centering	
  \begin{tabular}[h]{cc}
\hspace{0cm}\includegraphics[width=8cm]{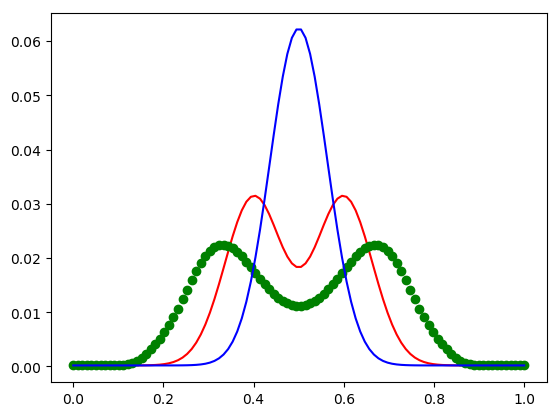}
&
\hspace{0cm}\includegraphics[width=8cm]{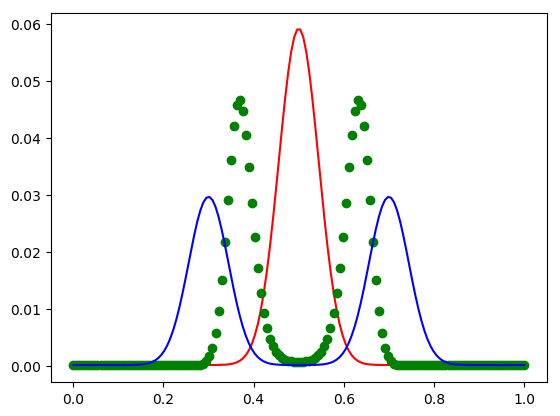}
\end{tabular}
\caption{On the left, a splitting experiment and on the right, a merging experiment.} 
\label{Extrapolations2}
\end{figure}
Note that this extrapolation scheme may proven useful in the development of higher-order schemes for the JKO algorithm.
\end{remark}

\section{Perspectives}

In this paper, we presented natural approaches to define cubic splines on the space of probability measures. We have presented a Monge formulation and its Kantorovich relaxation on the path space as well as their corresponding reduction on minimal cubic spline interpolation. We leave for future work theoretical questions such as the study of conditions under which the existence of a Monge map as a minimizer occurs, as well as the relaxation of cubic spline in the Wasserstein metric.
Our main contributions focus on the numerical feasibility of the minimization of the acceleration on the path space with marginal constraints. We have developed the entropic regularization scheme for the acceleration and shown simulations in 1D and 2D. Future work will address the 3D case which is out of reach with the methods presented in the first sections of this paper but possibly tackled with the semi-discrete method presented en Section \ref{SD}. In a similar direction, the application of this approach to the unbalanced case in the spirit of \cite{GeneralizedOT1} seems challenging due to the this dimensionality constraint and could be achieved within the semi-discrete setting.

In the Lagrangian setting, i.e. semi-discrete method, the extrapolation of a Wasserstein geodesic between $\rho_0$ and $\rho_1$ is obtained using three positions with the following formulation : let 
 $$\mathcal{Q}^N = \left\{ \sum^N_{j=1}   \frac1N  \delta_{\left(X^1_j,X^2_j,X^3_j\right) } \middle|  (X_j)_{j=1,\ldots,N}\in M^{n} \right\},$$ then 
\begin{equation}\label{SDVextr}
(SDextra) = \min_{\mathcal{Q}^N} \,\frac{1}{N}\sum_{j=1}^N \frac{d^2}{2}(X^1_j,X^2_j) + \frac{1}{N}\sum_{j=1}^N c(X^1_j,X^2_j,X^3_j) + \sum^2_{i=1} \frac{1}{2\epsilon^2} W_2^2\left(\sum^N_{j=1}   \frac1N  \delta_{X^i_j} , \rho_i \right),
\end{equation}
where $d$ is the distance on $M$ and $c(X^1_j,X^2_j,X^3_j)$ the cost of the cubic spline. 
In particular this formulation forces the curve to be a Wasserstein geodesic between $\rho_1$ and $\rho_2$, using the quadratic cost, and let free the final marginal. The implementation is completely similar as in Section \ref{SD} and the trajectory of each dirac masses is a straight line.

\appendix
\section{Proof of Theorem \ref{ThmRelaxation}}\label{Appendix}
The proof is a rewriting of the proof of \cite[Theorem 1.33]{santambrogio2015optimal}  when the initial and final spaces do not have the same dimension.
In particular we prove that transport plans concentrated on a graph of a map $T : \R^d \to  \R^p $ are dense into transport plans in $\R^d \times \R^p$ and deduce, taking $p= (n-1)d$, that for any continuous cost the multimarginal Kantorovich problem is the relaxation of the multimarginal Monge problem.

\begin{theorem}\label{relaxationmultimarginal}
Let $M=\R^d$ and $c:M^n \to \R$ be a continuous cost fonction. Let $(\rho_i)_{i\in {1,\ldots,n}}$ be $n$ probability measures on $M$. We define the Monge Problem $(M_c)$ as 
$$
(M_c)= \inf{  \int_{M}  c \left(x,T_2(x),\ldots , T_n(x)\right) \rho_1\,,} 
$$
over the set of map $\Pi_T = \left\{ T: M \to M^{n-1}, \, x \mapsto \left( T_i(x) \right)_{i=2,\ldots,n} \middle| \left(T_i \right)_*(\rho_1) = \rho_i, \, , i =2,\ldots,n\right\}$.
The Kantorovich problem $(K_c)$ is defined by 
$$
(K_c)= \inf{  \int_{M^n}  c \left(x_1,\ldots,x_n \right) \pi\left(x_1,\ldots,x_n \right) \,,} 
$$
over the set of plan $\Pi = \left\{ \pi \in \mathcal{P}(M^n) \middle|   (p_i)_*(\pi) = \rho_i, \, i =1,\ldots,n \right\}$, 
where $p_i$ is the projection of the $i^{\text{th}}$ factor. Then, if all $(\rho_i)_{i\in {1,\ldots,n}}$ have compact support and $\rho_1$ is atomless there holds $(M_c)=(K_c)$.
\end{theorem}
In order to prove Theorem \ref{relaxationmultimarginal} we first remark that \cite[Corrollary 1.29 and Theorem 1.32 ]{santambrogio2015optimal} have their multimarginal counterpart. 

\begin{lemma}\label{transportblock}
Let $\mu \in \mathcal{P}(\R^d)$ be atomless measure and $\nu \in \mathcal{P}(\R^p)$, then there exists a transport map $T: \R^d \to \R^p$ such that $T_* \mu =\nu$.
\end{lemma}
\begin{proof}[Proof of Lemma \ref{transportblock}]
Let  $\sigma_d : \R^d \to \R$ (resp $\sigma_p : \R^p \to \R$) be an injective Borel map with Borel inverse (see \cite[Lemma 1.28]{santambrogio2015optimal} for instance for a very simple proof of existence in this case).
Since $\mu$ is atomless $({\sigma_d})_*\mu$ is also atomless. Let $t: \R \to \R$ be the optimal transport map from $({\sigma_d})_*\mu$ to $ ({\sigma_{p}})_*\nu$ for the quadratic cost. 
 $t_*\left( ({\sigma_d})_*\mu \right) =  \left( {\sigma_{p}}\right)_*\nu $. Thus $T= \sigma_{p}^{-1} \circ t \circ \sigma_d$ is a map pushing forward $\mu$ to $\nu$.  
\end{proof}

\begin{theorem}\label{densitemultimarginal}
With the notation of Theorem \ref{relaxationmultimarginal}, if the support of all $\rho_i$ are included in a compact domain then the set of plans $\Pi_T$ induced by a transport is dense, for the weak topology, in the set of plans $\Pi$ whenever $\rho_1$ is atomless.
\end{theorem}

\begin{remark}
Theorem \ref{densitemultimarginal} is in fact very general, one can consider M N be only Polish spaces for instance. 
Then there exists invertible Borel maps from M (resp N) to $[0,1]$. This is enough to obtain Lemma \ref{transportblock}. 
Then one just need to consider a uniformly small partition of $\Omega$ to prove the density Theorem \ref{densitemultimarginal}. 
\end{remark}
\begin{proof}[Proof of Theorem \ref{densitemultimarginal}]
Again the proof is based on \cite[Theorem 1.32]{santambrogio2015optimal}. 
In particular the strategy of the proof is to approach a transport plan by transport maps defined on small sets on which the measure is preserved.

We consider a compact domain $\Omega = \Omega_d \times \Omega_p \in (\R^d \times \R^p)$ and $\pi \in \mathcal{P}(\Omega_d \times \Omega_p)$ such that  $(p_{\R^d})_*(\pi)=\mu$ is atomless. 
For any $m$ set a partition of $\Omega_p $ (resp $\Omega_q$) into (disjoint) sets $K_{i,m}$ (resp $L_{j,m}$) with diameter smaller than $1/2m$. 
Then $C_{i,j,m} = K_{i,m}\times L_{j,m} $ is a partition of $\Omega$ into sets with diameter smaller than $1/m$. 
Let $\pi_{i,m}$ be the restriction of $\pi$ on $K_{i,m}\times \Omega_p$ and $\mu_{i,m} = (p_{\R^d})_*(\pi_{i,m})$ and $\nu_{i,m} = (p_{\R^d})_*(\pi_{i,m})$.
Since $\mu$ is atomless $\mu_{i,m}=\mu_{|K_{i,m}} $ is also atomless and thanks to Lemma \ref{transportblock} there exists $t_{i,m}$ such that $(t_{i,m})_* \mu_{i,m}=\nu_{i,m}$.
 By definition 
 \begin{equation}\label{martingale}
 \pi[C_{i,j,m}]=\pi_{i,m}[C_{i,j,m}]= \mu_{i,m}[K_{i,j}]\nu_{i,m}[L_{j,m}]= (\rm{Id},t_{i,m})_*(\mu_{i,m})([C_{i,j,m}])=(\rm {Id},t_{m})_*(\mu)[C_{i,j,m}],
 \end{equation} 
where $t_m$ is define on $\Omega$ by $ t_{|K_{i,m}}=t_{i,m}$. In particular $(t_m)_*(\mu)=\nu$. Equation \eqref{martingale} and the definition of the partition sets $C_{i,j,m}$ implies that $(\rm{Id},t_{m})_*(\mu)$
weakly converges toward $\pi$ as $m  +\infty$ (they give same masses to any set of the partition). See [Theorem 1.31]{santambrogio2015optimal} for instance. 
To finish the proof let us remark that we can set $p=d(n-1)$ then $\mu=\rho_1$ is atomless and $t_m:$ $ \R^d \to \R^{d(n-1)}$ defines $(t_{2,n},...,t_{n,m})$.
\end{proof}

\begin{proof}[Proof of Theorem \ref{relaxationmultimarginal} ]
The continuity of the cost $c$ and the density Theorem \ref{densitemultimarginal} implies that $(K_c)\leq (M_c)$.
Since the converse is always true we have  $(M_c)= (K_c)$. 
\end{proof}
\begin{remark}
Theorem \ref{ThmRelaxation} is a consequence of Theorem \ref{Appendix} since both the Monge and the Kantorovich (Definition \ref{Mongeformulation} and \ref{ThmKantorovich}) problems reduces on $M^n$ with the spline cost which is continuous (see Corollary \ref{ThmMultiMarginal} and \ref{ThmMultiMarginalmonge}.
\end{remark}

\section{Entropic Regularisation and Sinkhorn}\label{App2}

\subsection{ Entropic regularization and Sinkhorn algorithm}

The  linear programming problems (\ref{JND}-\ref{JCDk})  is extremely costly to solve numerically and a natural strategy, which has received a lot of attention recently
following the pionneering works of \cite{galichon} and \cite{Cut}  is to approximate these problems by  strictly convex ones by adding an entropic penalization.  
It has been used with good results on a number of multi-marginal optimal transport problems  \cite{Ben}  \cite{benamoudft} \cite{benamoueuler}.  
Here is a rapid and simplified description, see the references above for more details. \\

The regularized  problem is 

\begin{equation}
\label{JNDe} 
\min_{ T^\epsilon}  \sum_{a,b}   \{  C_{a,b}  \, T^\epsilon_{a,b} + \epsilon  \, T^\epsilon_{a,b} \,  \log(T^\epsilon_{a,b}) \} 
\end{equation}
It is strictly convex. Denoting $u^k_{\alpha_{j_k}, \beta_{j_k}}$ the Lagrange multipliers of the k constraints (\ref{JCDk}), we obtain 
the optimality conditions: 
\begin{equation}
\label{F} 
T^\epsilon_{a,b}  =  K_{a,b}  \, \Pi_{k =1}^N U^k_{j_k}  
\end{equation}
where 
\[   
U^k_{j_k} = e^{ \frac{1}{\epsilon} u^k_{\alpha_{j_k}, \beta_{j_k}} } \quad \quad   K_{a,b}   = e^{  - \frac{1}{\epsilon}   C_{a,b}  }
\]
Equation (\ref{F}) caracterize the optimal tensor as a scaling of the Kernel $K$ depending on the dual unknown $U^k$. 
Inserting this factorization into the constrains (\ref{JCDk}) the dual problem takes the form of the set of equations ( $\forall k  \in [1,n]  $) 
 
\begin{equation}
\label{Dual}  
U^k_{j_k}  =  \rho_{j_k} ( x_{\alpha_{j_k}, \beta_{j_k}} )    (  \sum_{a \setminus \{\alpha_{j_k} \}  , \, b \setminus \{\beta_{j_k} \}  } 
    K_{a,b}  \, \Pi_{ k' \in \{1,..n\}      \setminus k }  \, U^{k'}_{j_{k'}}      )^{-1}    
\end{equation}

Sinkhorn algorithm simply amounts to perform a Gauss-Seidel type iterative resolution of the system (\ref{Dual}) and therefore
consists in computing the sums on the right-hand side and then perform the (grid) point wise division.  

\subsection{Implementation} 

In dimension 2, each unknown $U_k$ has dimension $N_x^2$,  the cost of one full Gauss Seidel cycle, i.e. on Sinkhorn iteration on all unknowns,  will therefore be $n \times N_x^2 \times$ the cost
to compute the tensor matrix products in the denominator of (\ref{Dual}).  Remember that $n$ is the number of time steps with constraints 
and $N$ the total number of time steps.  The given tensor Kernel $K_{a,b}$ is a priori a large 
$N \times N_x \times N_x$ tensor with indices   $ {a,b}  = {\alpha_1,..\alpha_N, \beta_1,..,\beta_N}$. It can however advantageously 
be tensorized both  along dimensions and also margins.  
First, using (\ref{acost}-\ref{acostD}) we see that the Kernel is the product of smaller  tensors
\[
K_{a,b} = \Pi_{i=1,N-1}   K^0_{i-1,i,i+1}  ,   \mbox{ with }     K^0_{i-1,i,i+1}  :=  e^{  - \frac{1}{\epsilon \,d\tau^3 }   \| x_{\alpha_{i+1},\beta_{i+1} }+ x_{\alpha_{i-1},\beta_{i-1} } -2 \, x_{\alpha_{i},\beta_{i} } \|^2 } .
\]
Moreover as we chose to work on a cartesian grid at all time steps, $K^0$ tensorize again into 
\[
 K^0_{i-1,i,i+1} =  K^\alpha_{i-1,i,i+1} \,  K^\beta_{i-1,i,i+1}    \mbox{ with }  
  K^\alpha_{i-1,i,i+1}  := e^{  - \frac{h^2}{\epsilon \,d\tau^3 }   \| \alpha_{i+1} + \alpha_{i-1} -2 \alpha_{i}  \|^2 }  
 \]
 Finally our large kernel $K_{a,b}$ can be represented a the product of  $2\,(N-2)$ identical tensors of size $N_x \times N_x \times N_x$.  
Assuming a cubic cost $n^3$  for the multiplication of two $(n \times n)$  matrix, we see oru algorithm is  of order $O(N \, N_x^4)$ in dimension 2.


\bibliographystyle{plain}
\bibliography{references,MesPapiers,bibli_jd}   

\end{document}